\documentclass[11pt]{amsart}
\usepackage[T1]{fontenc}
\usepackage[utf8]{inputenc}
\usepackage[english]{babel}
\usepackage{amsmath,amssymb,amsthm,mathtools,mathrsfs}
\usepackage{microtype}
\usepackage{enumitem}
\usepackage{xurl}
\usepackage{xcolor}
\usepackage[colorlinks=true,linkcolor=blue,citecolor=blue,urlcolor=blue]{hyperref}
\mathtoolsset{showonlyrefs}
\numberwithin{equation}{section}

\newtheorem{theorem}{Theorem}[section]
\newtheorem{lemma}[theorem]{Lemma}
\newtheorem{proposition}[theorem]{Proposition}
\newtheorem{corollary}[theorem]{Corollary}
\newtheorem{theoremA}{Theorem}

\theoremstyle{definition}
\newtheorem{definition}[theorem]{Definition}

\newtheoremstyle{remarkbold}
  {3pt}{3pt}
  {\normalfont}
  {}
  {\bfseries}
  {.}
  {.5em}
  {}
\theoremstyle{remarkbold}
\newtheorem{remark}[theorem]{Remark}

\newcommand{\Prob}{\operatorname{Prob}}
\newcommand{\mdim}{\operatorname{mdim}}
\newcommand{\rank}{\operatorname{rank}}
\newcommand{\supp}{\operatorname{supp}}
\newcommand{\mult}{\operatorname{mult}}

\newcommand{\1}{\mathbf 1}
\def\ord{\mathrm{ord}}\def\rLocT{\operatorname{Loc}_{\mathcal D}^{T}}
\def\mesh{\mathrm{mesh}}
\def\rPG{\mathcal P_G}\def\rLocD{\mathrm{loc}_{\mathcal D}}\def\rOrdD{\mathrm{ord}_{\mathcal D}}%
\def\hps{\bigl((\widetilde{O_j})_{j=1}^J, (D_j)_{j=1}^J, \boldsymbol{\psi}\bigr)} \def\rDdim{\mathrm{ddim}}\def\bfpsi{\boldsymbol \psi}
\title[Generic Sharp Shift Embeddings for Amenable Group Actions]{Generic Sharp Shift Embedding for Amenable Group Actions with the Uniform Rokhlin Property and Comparison}
\author{Yonatan Gutman and Chunlin Liu}

\address{\vskip 2pt \hskip -12pt Yonatan Gutman}

\address{\hskip -12pt Institute of Mathematics, Polish Academy of Sciences, ul. Śniadeckich 8, 00-656 Warszawa, Poland}

\email{gutman@impan.pl}

\medskip
\address{\vskip 2pt \hskip -12pt Chunlin Liu}

\address{\hskip -12pt School of Mathematical Sciences, Dalian University of Technology, Dalian, 116024, P.R. China and Institute of Mathematics, Polish Academy of Sciences, ul. Śniadeckich 8, 00-656 Warszawa, Poland}

\email{chunlinliu@mail.ustc.edu.cn}
\date{}
\keywords{mean dimension, amenable group action, uniform Rokhlin property, comparison, cubical shift, generic embedding}
\thanks{Y.G. was partially supported by the National Science Centre (Poland) grant 2020/39/B/ST1/02329.
C.L. was supported by Liaoning Provincial Natural Science Foundation
 Doctoral Research Start-up Project (Project No.~2026-BS-0032), the
Postdoctoral Fellowship Program and China Postdoctoral Science Foundation under Grant Number BX20250067, and the China Postdoctoral Science Foundation under Grant Number 2025M773074.}
\begin{document}

\begin{abstract}
Naryshkin recently proved that a free action of a countably infinite
discrete amenable group $G$ on a compact metrizable space $X$ with
the uniform Rokhlin property and comparison (URPC) admits an
equivariant embedding into $([0,1]^m)^G$ whenever
$\mdim(X,G)<m/2$.
We strengthen this result from existence to genericity by a
substantially different method: the set of observables
$f\in C(X,[0,1]^m)$ whose orbit maps are equivariant embeddings
into $([0,1]^m)^G$ is a dense $G_\delta$ subset of
$C(X,[0,1]^m)$.
The proof combines an adaptation of the Kakutani skyscraper
construction used by Gutman--Tsukamoto with a new polynomial
perturbation method, departing from the predominantly linear
perturbation approaches used in earlier equivariant embedding literature.
\end{abstract}
\maketitle

\section{Introduction}

A fundamental research endeavor in the theory of topological dynamical systems originated in  Auslander's famous 1988 question \cite{auslander1988} 
: \textit{Does every minimal homeomorphism of a compact
metrizable space admit an equivariant embedding into the shift on
$[0,1]^{\mathbb Z}$?}  Mean dimension supplied the decisive obstruction. Introduced by Gromov
\cite{gromov1999} in 1999 and developed systematically by Lindenstrauss and Weiss
\cite{lindenstraussweiss2000}, it measures the average amount of
topological dimension per iterate. The shift on $([0,1]^m)^{\mathbb Z}$
has mean dimension $m$, and mean dimension is monotone under equivariant
embeddings. Lindenstrauss and Weiss constructed a minimal system of mean
dimension greater than one, thereby answering Auslander's question
negatively. Lindenstrauss then established a striking partial converse:
a minimal system with $\mdim(X,T)<\frac{m}{36}$ \textit{generically}\footnote{ Genericity means that the set of
observables $f\in C(X,[0,1]^m)$ for which the orbit map
$x\mapsto(f(T^n x))_{n\in\mathbb Z}$ is an embedding
is a dense $G_\delta$ subset of $C(X,[0,1]^m)$,
equipped with the uniform topology.} embeds equivariantly into the $m$-cubical
shift \cite{lindenstrauss1999}. In the same paper, he asked for the largest
constant $c$ for which the bound $\mdim(X,T)<cm$ guarantees an
equivariant embedding. Gutman and Tsukamoto \cite{gutmantsukamoto2020} solved the problem by proving the existence\footnote{Unlike the result of \cite{lindenstrauss1999}, the result in \cite{gutmantsukamoto2020}  is an existence result, not establishing genericity. The generic embedding result for minimal $\mathbb{Z}$-actions was first proven in \cite{levin2023finite}.} of an equivariant embedding for any $c<\frac{1}{2}$. As Lindenstrauss and Tsukamoto \cite{lindenstrausstsukamoto2014} had previously constructed
minimal systems of mean dimension exactly $m/2$ not embeddable into
the $m$-cubical shift, the constant $1/2$, and the strict inequality, were shown to be optimal.

 The marker property introduced in \cite{gutman2015} based on related notions of \cite{downarowicz2006} and \cite{krieger1982} is  a natural generalization of minimality in the context of the embedding problem.
For a group action $G\curvearrowright X$, it requires that for every
finite $F\subset G$ there be an open set $U$ whose translates
$gU$, $g\in F$, are pairwise disjoint and whose full family of
translates covers $X$. Every free minimal action has this property. Following the development of dynamical tiling methods in
\cite{gutmanlindenstrausstsukamoto2016}, Gutman, Qiao, and Tsukamoto
proved the sharp embedding theorem for $\mathbb Z^k$-actions with the
marker property \cite{gutmanqiaotsukamoto2019}.

In view of the fundamental role of the Rokhlin lemma in coding
arguments in ergodic theory (\cite{K70,ornstein1970bernoulli,O74}), the uniform Rokhlin property (URP) introduced by Niu \cite{niu2022} in his study
of crossed-product $C^*$-algebras associated with free minimal
actions of amenable groups\footnote{This is one of many manifestations of the interaction
between mean dimension theory and the structure and classification
of $C^*$-algebras; see, for example,
\cite{kerr2020,kerrszabo2020,niu2022} and the discussion in
\cite{naryshkin2024}.},  offers a particularly natural setting for the embedding problem\footnote{There have, however, also been deep recent developments
for systems beyond this class; see
\cite{dranishnikovlevin2025,meyerovitch2026,shi2026sharp}.}. It stipulates the existence of finite families of towers with arbitrarily invariant shapes and mutually disjoint levels, whose union covers all but an arbitrarily small part of the space, uniformly over all invariant probability measures. Naryshkin proved
that a free action of a countably infinite discrete amenable group $G$
with URP and  \textit{comparison}\footnote{Essentially, comparison allows sets that are smaller in every invariant measure to be moved
piece by piece into larger ones by group translates, see Definition \ref{def:comparison}.} (URPC), satisfying $\mdim(X,G)<m/2$, embeds into $([0,1]^m)^G$
\cite[Theorem~A]{naryshkin2024}. Note that for free abelian-group actions,
URPC is equivalent to the marker property (\cite[Corollary~E]{naryshkin2024}). However, the question whether general amenable free minimal actions
satisfy URPC remains open.\footnote{Very recently, it was shown in
\cite{GlasnerLiu2026} that every free action of a countably infinite
amenable group on a finite-dimensional compact metrizable space has
URPC.}

In his proof of the embedding theorem, Naryshkin used an ingenious
construction that encodes an entire sequence of URPC covers in
a single continuous observable \cite[Lemmas~6.6--6.7]{naryshkin2024}. The Baire-category argument
is then carried out in a closed proper  subclass of observables constrained by that
function.
The passage from existence to genericity requires one to perform the Baire-category argument in the \textit{entire} observable space.  For
$f\in C(X,[0,1]^m)$, denote its \textit{orbit map} by
$$
f^G:X\to([0,1]^m)^G,
 \qquad f^G(x)_g=f(gx),\quad g\in G.
$$
We state our main theorem.

\begin{theoremA}\label{thm:main}
Let a countable\footnote{If $G$ is finite, the conclusion holds assuming only freeness
by
\cite[Theorem~1.1]{gutmanlevinmeyerovitch2025}.
Henceforth, we assume that $G$ is infinite.} discrete amenable group $G$ act freely by
homeomorphisms on a nonempty compact metrizable space $X$, and assume that
the action has URPC. If
$
 \mdim(X,G)<\frac m2,
$
then
$$
 \mathcal E_m(X,G)
 :=
 \bigl\{
 f\in C(X,[0,1]^m):
 f^G\text{ is a topological embedding}
 \bigr\}
$$
is a dense $G_\delta$ subset of $C(X,[0,1]^m)$.
\end{theoremA}
In particular, the conclusion applies to free abelian-group actions
with the marker property, and hence to free minimal abelian-group
actions, under the stated mean-dimension bound. Sharpness as a uniform
theorem already follows from the minimal $\mathbb Z$-examples of
\cite{lindenstrausstsukamoto2014}.
 This also gives
a dynamical counterpart of the generic Menger--N\"obeling theorem:
for a compact metrizable space $K$ with $\dim K<m/2$, embeddings into
$[0,1]^m$ form a dense $G_\delta$ subset of $C(K,[0,1]^m)$
\cite{hurewiczwallman1941,engelking1978}.

Our proof differs significantly from Naryshkin's proof. Ultimately generic embedding is achieved through a novel method detailed in Section \ref{section:main}.
The key ingredient, proven in that section, is Proposition~\ref{prop:finite-family-localizer},
which we call \emph{finite-window localization up to a buffer}.
Its role can be described as follows. From a \textit{URPC witness} we
construct finitely many compact sets $K_\ell$ with open
neighborhoods $O_\ell$. Suitable translates of the $K_\ell$
cover $X$, while the corresponding translates of the $O_\ell$
are contained in tower levels or open sets covering the remainder.

Starting from a suitable partition-of-unity approximation $f_*$
of an arbitrary observable, we aim to perturb $f_*$ arbitrarily slightly
to $q$ so that, for a finite set $F\subset G$,
\begin{equation}\label{eq:intro-localization}
 q^F(K_\ell)\cap q^F(X\setminus O_\ell)=\emptyset,
 \qquad q^F(x):=\big(q(gx)\big)_{g\in F}.
\end{equation}
Thus, whenever $x\in K_\ell$, any point with the same observations
as $x$ on the finite window $F$ must belong to $O_\ell$.
Once this localization has been achieved, applying it to suitable
translates shows that points with identical full orbit observations
lie in a common tower level or in a common member of the open
cover of the remainder.
This allows the subsequent separation argument to be carried
out within these regions. Indeed, we use general-position perturbations of orbit blocks,
as in the \textit{Kakutani skyscraper decomposition} technique of Gutman--Tsukamoto   
\cite{gutmantsukamoto2014},
and combine the resulting block maps into a single continuous
observable by placing their coordinates along the corresponding
tower levels. The buffer between $K_\ell$ and
$X\setminus O_\ell$ makes the localization stable under
the additional (sufficiently small) perturbations. 

The main challenge is to achieve \eqref{eq:intro-localization} by an arbitrarily small perturbation. Departing from the predominantly linear perturbation approaches
used in earlier equivariant embedding arguments, we introduce \textit{polynomial
perturbations of partition-of-unity maps}. Indeed, if
$f_* =\sum_j v_j\psi_j$,  where $v_j\in(0,1)^m$ and
$\boldsymbol{\psi}=(\psi_1,\ldots,\psi_J)$ is a partition of unity, the perturbations have the form
$$
 q_\Theta(z)_i=f_*(z)_i+
 \langle\Xi(z),\theta^{(i)}\rangle,
 \qquad 1\leq i\leq m,
$$
where $\Xi$ consists of translated cutoffs and monomials of bounded
degree in finitely many translates of
$\boldsymbol{\psi}$.\footnote{Polynomial perturbations have been used in a non-equivariant Euclidean-space setting in the context of the \textit{Takens-type delay embedding theorems}, e.g. see \cite{SYC91,baranski2024prediction}.} The coefficients $\Theta$ may be
chosen arbitrarily close to zero. The proof is completed by a combination of combinatorial
rank estimates and semialgebraic dimension arguments.

\subsection*{Organization of the paper}
Section~\ref{sec:preliminaries} introduces notation and background material
used throughout the paper.
In Section~\ref{sec:global-coordinates}, we develop the
low-dimensional coordinate systems needed for the perturbation
argument.
Section~\ref{sec:global-approximation} establishes the finite-scale
approximation theorem and derives Theorem~\ref{thm:main} by a
Baire-category argument, assuming the \textit{finite-window localization up to a buffer}
result.
Finally, Section~\ref{sec:recognition-proof} proves this localization
result and completes the proof.

\section{Preliminaries}\label{sec:preliminaries}

Throughout the paper, $G$ is a countably infinite discrete amenable group acting on the left by homeomorphisms on  a
nonempty compact metrizable space $X$, and $e$ denotes the identity
of $G$. We fix $m\in\mathbb N$ and a compatible metric $\rho$ on
$X$. For subsets $A,B\subset X$, we write
$
 A\Subset B
$
if
$
 \overline{A}\subset B.
$ All inclusions are understood to allow equality.

We write $\rPG(X)$ for the space of $G$-invariant Borel
probability measures on $X$. We equip $\mathbb R^m$ and its finite products with the maximum norm,
denoted by $\|\cdot\|$, and function spaces with the corresponding
uniform norm $\|\cdot\|_\infty$.

\subsection{Basic notation}
For $q\in C(X,[0,1]^m)$ and any subset $D\subset G$, finite or infinite, write
$q^D(x):=(q(gx))_{g\in D}$.
Every finite product of Euclidean cubes is equipped with the maximum
norm. We equip $([0,1]^m)^G$ with the left shift
$(h\cdot z)_g=z_{gh},$ $g,h\in G$,
and define the equivariant orbit map
\[
 q^G:X\longrightarrow([0,1]^m)^G,
 \qquad
 q^G(x)_g=q(gx).
\]
Then
$q^G(hx)=h\cdot q^G(x)$.

More generally, if
$\xi=(\xi_\lambda)_{\lambda\in L}$ is a finite tuple of continuous
functions and $D\subset G$ is finite, put
\[
 \xi^D(x)
 :=
 \bigl(\xi_\lambda(dx)\bigr)_{(d,\lambda)\in D\times L}.
\]
For a single vector-valued map this reduces to
$\xi^D(x)=(\xi(dx))_{d\in D}$.

For a finite indexed family
$\mathcal C=(C_j)_{j=1}^J$ of subsets of a space $X$, define
\[
 \operatorname{loc}_{\mathcal C}(x)
 :=
 \{j:x\in C_j\},
 \qquad
 \mult_{\mathcal C}(x)
 :=
 |\operatorname{loc}_{\mathcal C}(x)|.
\]
The multiplicity of $\mathcal C$ is
\[
 \mult(\mathcal C)
 :=
 \sup_{x\in X}\mult_{\mathcal C}(x).
\]
If $\mathcal C$ is a finite cover, define
\[
 \ord_{\mathcal C}(x)
 :=
 \mult_{\mathcal C}(x)-1,
 \qquad
 \ord(\mathcal C)
 :=
 \mult(\mathcal C)-1
 =
 \sup_{x\in X}\ord_{\mathcal C}(x).
\]

For a finite open cover $\mathcal U$ of a compact metrizable space, set
\[
 D(\mathcal U)
 :=
 \min\bigl\{\ord(\mathcal V):
       \mathcal V\text{ is a finite open cover refining }\mathcal U\bigr\}.
\]
For a cover $\mathcal U$ of $(X,\rho)$, write
\[
 \mesh(\mathcal U):=\sup_{U\in\mathcal U}\operatorname{diam}_\rho(U).
\]
A map $H:(X,\rho)\to Y$ is an $\eta$-\textbf{embedding} if
$H(x)=H(y)\Longrightarrow \rho(x,y)<\eta$.

Let $H:Y\to Z$ be a map and $C\subset Y$. We say that $H$
\textbf{encodes} $C$ if
$$H^{-1}(H(C))=C.$$
We say that $H$ \textbf{encodes a finite family} if it encodes every member of that family.

\subsection{Partitions of unity and simplicial notation}\label{subsec:ve}

Let
\[
 \boldsymbol{\psi}=(\psi_1,\ldots,\psi_J):X\to\Delta^{J-1}:=\left\{(z_1,\ldots ,z_J)\in [0,1]^J:\sum_{j=1}^Jz_j=1\right\}
\]
be a partition of unity.  For a finite set
$E\subset G$, denote
\[
 \boldsymbol{\psi}^{E}(x):=
 (\boldsymbol{\psi}(tx))_{t\in E}.
\]
Given an indexed open cover
$(\widetilde{O_j})_{j=1}^J$, we say that $\boldsymbol{\psi}$ is \textbf{honest} for
$(\widetilde{O_j})_{j=1}^J$ if
$\{\psi_j>0\}=\widetilde{O_j}$, $1\le j\le J$.
Suppose moreover that a closed cover
$(D_j)_{j=1}^J$ is fixed and
$\supp\psi_j\subset D_j$, $1\le j\le J$. We call the triple
\[
\mathcal H:=\hps
\]
an \textbf{honest partition structure}.

For a nonempty $I\subset\{1,\ldots,J\}$, let $\Delta_I$ be the
coordinate face of $\Delta^{J-1}$ spanned by the vertices indexed by $I$. Write $\mathcal D:=(D_j)_{j=1}^J$.
Since $\supp\psi_j\subset D_j$, 
\[
 \boldsymbol\psi(x)\in\Delta_{\operatorname{loc}_{\mathcal D}(x)},
 \qquad
 \dim\Delta_{\operatorname{loc}_{\mathcal D}(x)}
 =\ord_{\mathcal D}(x).
\]

For later use, we also fix the polynomial coordinates on the simplex.
For $d\in\mathbb N_0$, let
\[
 \mathcal A_{J,d}
 :=
 \left\{
 \alpha=(\alpha_1,\ldots,\alpha_J)\in\mathbb N_0^J:
 |\alpha|:=\alpha_1+\cdots+\alpha_J\le d
 \right\},
\]
and denote
$M_{J,d}:=|\mathcal A_{J,d}|=\binom{J+d}{d}$.
For $z=(z_1,\ldots,z_J)\in\Delta^{J-1}$ and
$\alpha\in\mathcal A_{J,d}$, write
$z^\alpha:=z_1^{\alpha_1}\cdots z_J^{\alpha_J}$.
We define the \textbf{degree-$d$ Veronese map} by
\[
 \Phi_d:\Delta^{J-1}\longrightarrow\mathbb R^{M_{J,d}},
 \qquad
 \Phi_d(z):=(z^\alpha)_{\alpha\in\mathcal A_{J,d}}.
\]
Thus every polynomial $P$ on $\mathbb R^J$ of degree at most $d$
can be written in the form
\begin{equation}\label{eq:v-rep}
 P(z)=\langle a_P,\Phi_d(z)\rangle\qquad\text{for some $a_P\in\mathbb R^{M_{J,d}}$.}    
\end{equation}

We shall also use the following standard interpolation terminology.

\begin{definition}
Let $Y$ be a set and let $k\ge1$. A finite family of functions
$h_1,\ldots,h_M:Y\to\mathbb R$ is called \textbf{$k$-interpolating on $Y$}
if, for every $1\le q\le k$, every collection of pairwise distinct
points $z_1,\ldots,z_q\in Y$, and every
$(u_1,\ldots,u_q)\in\mathbb R^q$, there exists
$\theta=(\theta_1,\ldots,\theta_M)\in\mathbb R^M$ such that
\[
 \sum_{\alpha=1}^M\theta_\alpha h_\alpha(z_i)=u_i,
 \qquad 1\le i\le q.
\]
\end{definition}

By the classical polynomial interpolation argument
\cite[Lemma~4.1(1)]{SYC91}; see also \cite[Section 1.2, eq. (1.9)]{gasca2000polynomial}, and \eqref{eq:v-rep}, 
\begin{equation}\label{eq:interpolation}
     \text{the family }\{z\mapsto z^\alpha:\alpha\in\mathcal A_{J,d}\}\text{ is $(d+1)$-interpolating on $\Delta^{J-1}$.}
\end{equation}

\subsection{Uniform Rokhlin property and comparison}
An \textbf{open castle} is a finite family
$
 \{(S_i,V_i)\}_{i=1}^r,
$
where each $S_i\subset G$ is finite and each $V_i\subset X$ is
open, such that the levels
$
 sV_i,
$ $ 1\le i\le r, s\in S_i,$
are pairwise disjoint. Each pair $(S_i,V_i)$ is a \textbf{tower}, $S_i$ is its
\textbf{shape}, $V_i$ is its \textbf{base}, and the sets $sV_i$, $s\in S_i$,
are its \textbf{levels}. The closed set
$X\setminus\bigcup_{i=1}^r S_iV_i$ is the \textbf{remainder} of the castle.

For a finite set $K\subset G$ and $\alpha>0$, a nonempty finite set
$S\subset G$ is called $(K,\alpha)$-\textbf{invariant} if
$
 \left|\bigcap_{k\in K}kS\right|
 >
 (1-\alpha)|S|.
$

For a finite set $L\subset G$, an open set $U\subset X$ is
$L$-\textbf{free} if the family $\{\ell U:\ell\in L\}$ is pairwise
disjoint.  It is a \textbf{sweep-out set} if $MU=X$ for some finite
$M\subset G$.  The action has the \textbf{marker property} if it admits
an $L$-free sweep-out set for every finite $L\subset G$.

The following property of finite invariance is
a standard consequence; see, for example,
\cite[Section~4.1]{kerrli2016}.

\begin{lemma}
\label{lem:finite-invariance-propagation}
Let $A,E,K\subset G$ be finite, with $e\in A$ and
$E\ne\emptyset$, and let $\varepsilon_A,\varepsilon_E>0$ and
$\delta\in(0,1)$.  Then there exist a finite set $K_0\subset G$ and
$\delta_0\in(0,1)$ such that every nonempty finite
$(K_0,\delta_0)$-invariant set $S\subset G$ satisfies
$|AS|<(1+\varepsilon_A)|S|$ and
$|ES|<(1+\varepsilon_E)|S|$, and $AS$ is
$(K,\delta)$-invariant.
\end{lemma}

For a closed set $A\subset X$ and an open set $B\subset X$, we write
$
 A\prec B
$
if there exist a finite open cover $\mathcal U$ of $A$ and elements
$g_U\in G$, $U\in\mathcal U$, such that the sets
$
 \{g_UU:U\in\mathcal U\}
$
are pairwise disjoint subsets of $B$.
\begin{definition}[{Kerr \cite[Definition~3.2]{kerr2020}}] \label{def:comparison}
The action $G\curvearrowright X$ has \textbf{comparison} if
$
 A\prec B
$
whenever $A\subset X$ is closed, $B\subset X$ is open, and
$\mu(A)<\mu(B)$ for every $\mu\in  \rPG{}(X)$.
\end{definition}
 
\begin{definition}[{Niu \cite{niu2022}}]
The action $G\curvearrowright X$ has the
\textbf{uniform Rokhlin property} (URP) if, for every finite
$K\subset G$ and every $\varepsilon>0$, there is an open castle
$
 \{(S_i,V_i)\}_{i=1}^n
$
such that every $S_i$ is $(K,\varepsilon)$-invariant and
\[
\sup_{\mu\in  \rPG{}(X)}\mu\left(
 X\setminus\bigsqcup_{i=1}^n S_iV_i
 \right)<\varepsilon.
\]
\end{definition}

As the remainder is closed, this invariant-measure formulation agrees
with Naryshkin's upper-density formulation; see
\cite[Subsection~2.1]{naryshkin2024}.
\begin{remark}\label{rem:urp-almost-free}
If an action has URP, then it is free almost everywhere for every invariant probability measure. This assertion does not require topological freeness. Write $\operatorname{Fix}(g):=\{x\in X:gx=x\}$.
Fix $g\ne e$ and $\varepsilon>0$, and apply URP with $K=\{e,g^{-1}\}$. For each resulting shape $S_i$,  denote
$C_i:=S_i\cap g^{-1}S_i$. Then
$|S_i\setminus C_i|<\varepsilon|S_i|$.
If $s\in C_i$ and $z\in sV_i$ satisfies $gz=z$, then
$z=gz\in gsV_i$, contradicting the disjointness of the distinct
levels $sV_i$ and $gsV_i$. Thus, for every $\mu\in \mathcal P_G(X)$,
\[
\begin{aligned}
 \mu(\operatorname{Fix}(g))
 &\le \mu\left(X\setminus\bigsqcup_i S_iV_i\right)
       +\sum_i|S_i\setminus C_i|\mu(V_i)\\
 &<\varepsilon+\varepsilon\sum_i|S_i|\mu(V_i)
 \le 2\varepsilon.
\end{aligned}
\]
Letting $\varepsilon\downarrow0$ gives
$\mu(\operatorname{Fix}(g))=0$. Since $G$ is countable, the action
is free $\mu$-almost everywhere for every invariant $\mu$.
\end{remark}

\begin{definition}[Naryshkin \cite{naryshkin2024}]
The action $G\curvearrowright X$ has
\textbf{URPC} if it has both the uniform Rokhlin property and comparison.
\end{definition}

We shall use the following finite-data formulation.

\begin{definition}[{Naryshkin \cite[Definition~3.24]{naryshkin2024}}]\label{def:URPC-witness}
Let $K\subset G$ be finite and let $\alpha>0$. A
$(K,\alpha)$-\textbf{URPC witness} consists of finite data
\[
 \mathfrak W
 =
 \left(
 \{S_i\}_{i=1}^r,
 \{V_i\}_{i=1}^r,
 \{g_j\}_{j=1}^q,
 \{U_j\}_{j=1}^q
 \right)
\]
such that:
\begin{enumerate}[label=\textup{(U\arabic*)},leftmargin=*]
\item
$\{(S_i,V_i)\}_{i=1}^r$ is an open castle such that every $S_i$ is
$(K,\alpha)$-invariant;

\item
the open sets $U_1,\ldots,U_q$ cover the castle remainder
$
 X\setminus\bigsqcup_{i=1}^r S_iV_i;
$

\item
the index set $\{1,\ldots,q\}$ is partitioned into subsets
$J_{i,c}$, where
$
 1\le i\le r,
$ and $
 1\le c\le\lfloor\alpha|S_i|\rfloor,
$
such that, whenever $j\in J_{i,c}$,
$
 g_jU_j\subset V_i,
$
and, for every fixed pair $(i,c)$, the sets
$
 \{g_jU_j:j\in J_{i,c}\}
$
are pairwise disjoint.
\end{enumerate}
Empty index ranges and empty sets $J_{i,c}$ are allowed.
The elements $g_j$ are called \textbf{transporters}, and the integers $c$ are called \textbf{colors}.

\end{definition}

\begin{proposition}[{Naryshkin \cite[Theorem~3.20(iii)]{naryshkin2024}}]
\label{prop:urpc-witness-equivalence}
A free action has URPC if and only if it admits a $(K,\alpha)$-URPC
witness for every finite $K\subset G$ and every $\alpha>0$.
\end{proposition}

\subsection{Amenability and mean dimension}

A sequence $(F_n)_{n\ge1}$ of nonempty finite subsets of $G$ is a
\textbf{F{\o}lner sequence} if
\[
 \lim_{n\to\infty}
 \frac{|gF_n\triangle F_n|}{|F_n|}
 =0
 \qquad\text{for every }g\in G.
\]
Since $G$ is amenable, such sequences exist.

For a finite open cover $\mathcal U$ of $X$ and a nonempty finite
set $F\subset G$, denote
\[
 \mathcal U^F
 :=
 \bigvee_{g\in F}g^{-1}\mathcal U
 =
 \left\{
 \bigcap_{g\in F}g^{-1}U_g:
 U_g\in\mathcal U
 \right\}.
\]

Let $(F_n)_{n\ge1}$ be a F{\o}lner sequence.  The mean dimension of
the cover $\mathcal U$ is
\[
 \mdim(\mathcal U,G)
 :=
 \lim_{n\to\infty}
 \frac{D(\mathcal U^{F_n})}{|F_n|}.
\]
By the Ornstein--Weiss lemma \cite{ornsteinWeiss1987,lindenstraussweiss2000}, the limit exists and is independent of
the choice of the F{\o}lner sequence.  The \textbf{mean dimension} of
the action is
$$\mdim(X,G) := \sup_{\mathcal U}\mdim(\mathcal U,G),$$
where the supremum is taken over all finite open covers
$\mathcal U$ of $X$.

Let $\mathcal U$ and $\mathcal V$ be covers of a topological space $X$.
We say that $\mathcal V$ \textbf{refines} $\mathcal U$, and write
$
\mathcal V\succ\mathcal U,
$
if for every $V\in\mathcal V$ there exists $U\in\mathcal U$ such that
$
V\subset U.
$

Meyerovitch \cite{meyerovitch2026} defined the following invariant, called the \textbf{dynamical dimension}:
\[
  \rDdim(X,G)
: =
 \sup_{\mathcal A}
 \inf_{\mathcal B\succ\mathcal A}
 \sup_{\mu\in  \rPG{}(X)}
 \int_X
  \ord_{\mathcal B}(x)
 \,d\mu(x),
\]
where $\mathcal A$ and $\mathcal B$ range over finite open covers
of $X$. By \cite[Theorem~9.2]{meyerovitch2026}, URP implies
\begin{equation}\label{eq:dynamical-mean-dimension}
  \rDdim(X,G)=\mdim(X,G).    
\end{equation}

\subsection{Semialgebraic sets}
For a topological space $Z$, $\dim Z$ denotes its Lebesgue covering
dimension.

We recall facts about semialgebraic sets that will be used below; see
\cite[Sections~2.2 and~2.8]{bcr1998}.

A subset $Z\subset\mathbb R^N$ is called \textbf{semialgebraic} if it
is a finite union of sets of the form
\[
 \left\{
 x\in\mathbb R^N:
 p_1(x)=\cdots=p_r(x)=0,\ 
 q_1(x)>0,\ldots,q_s(x)>0
 \right\},
\]
where $p_1,\ldots,p_r,q_1,\ldots,q_s$ are real polynomials.

If $Z\subset\mathbb R^N$ is semialgebraic, a map
$
 f:Z\to\mathbb R^M
$
is called \textbf{semialgebraic} if its graph
\[
 \Gamma_f
 :=
 \{(z,f(z)):z\in Z\}
 \subset\mathbb R^{N+M}
\]
is semialgebraic.  A map with values in a finite-dimensional vector
space, or in a matrix space, is semialgebraic if its coordinate
functions are semialgebraic.

\begin{lemma}
\label{lem:basic-semialgebraic-facts}
Let $X,Y$ be semialgebraic sets.
\begin{enumerate}[label=\textup{(\roman*)},leftmargin=*]

\item If
$
 X=X_1\cup\cdots\cup X_N
$
with each $X_i$ semialgebraic, then
$$\dim X=\max_{1\le i\le N}\dim X_i.$$

\item The product $X\times Y$ is semialgebraic and
$
 \dim(X\times Y)=\dim X+\dim Y.
$

\item The closure $\overline X$ is semialgebraic and
$
 \dim\overline X=\dim X.
$

\item If $f:X\to Y$ is semialgebraic, $A\subset X$ is
semialgebraic, and $B\subset Y$ is semialgebraic, then $f(A)$ and
$f^{-1}(B)$ are semialgebraic and
$$\dim f(A)\le\dim A.$$

\item Let $f:X\to Y$ be a continuous semialgebraic map.  If every
nonempty fibre $f^{-1}(y)$ has dimension at most $q$, then
$$\dim X\le\dim f(X)+q \le\dim Y+q.$$

\item If $X\subset\mathbb R^N$ and
$
 \dim X<N,
$
then
\[
 \operatorname{int}_{\mathbb R^N}(\overline X)=\emptyset.
\]
That is, $X$ has nowhere dense closure. 

\end{enumerate}
\end{lemma}
\begin{proof}
 (i) and (ii) are
\cite[Proposition~2.8.5 (i) (ii)]{bcr1998}, respectively.

For (iii), the semialgebraicity of $\overline X$ follows from
\cite[Proposition~2.2.2]{bcr1998}, while
$
 \dim\overline X=\dim X
$
is \cite[Proposition~2.8.2]{bcr1998}.

For (iv), the semialgebraicity of $f(A)$ and $f^{-1}(B)$ follows
from \cite[Proposition~2.2.7]{bcr1998}, and
$
 \dim f(A)\le \dim A
$
is \cite[Theorem~2.8.8]{bcr1998}.

For (v), apply Hardt's semialgebraic triviality theorem
\cite[Theorem~9.3.2]{bcr1998}.  There is a finite semialgebraic
partition
$
 f(X)=\bigsqcup_{\ell=1}^s Y_\ell
$
such that, for each $\ell$,
there is a semialgebraic homeomorphism
$h_\ell:f^{-1}(Y_\ell)\to Y_\ell\times F_\ell$
such that $\operatorname{pr}_1\circ h_\ell=f|_{f^{-1}(Y_\ell)}$.
Here $\operatorname{pr}_1$ is projection onto $Y_\ell$. In particular,
$h_\ell$ identifies each fibre over $y\in Y_\ell$ with $\{y\}\times F_\ell$.
Hence
$\dim F_\ell\le q$.
Using (i) and (ii), we obtain
\[
 \begin{aligned}
 \dim X
 &=
 \max_{1\le \ell\le s}\dim f^{-1}(Y_\ell)
 =
 \max_{1\le \ell\le s}
 \bigl(\dim Y_\ell+\dim F_\ell\bigr)\\
 &\le
 \max_{1\le \ell\le s}\dim Y_\ell+q=
 \dim f(X)+q\le \dim Y+q.
 \end{aligned}
\]

Finally, suppose that $X\subset\mathbb R^N$ and
$\dim X<N$.  By (iii),
$\dim\overline X=\dim X<N$.
If $\operatorname{int}_{\mathbb R^N}(\overline X)$ were nonempty,
then it would be a nonempty open semialgebraic subset of
$\mathbb R^N$, and therefore
\[
 \dim\operatorname{int}_{\mathbb R^N}(\overline X)=N
\]
by \cite[Proposition~2.8.4]{bcr1998}.  Since
$
 \operatorname{int}_{\mathbb R^N}(\overline X)\subset\overline X,
$
this would imply
$
 N\le \dim\overline X,
$
a contradiction.  Thus
\[
 \operatorname{int}_{\mathbb R^N}(\overline X)=\emptyset.
\]
Since $\overline X$ is closed, it follows that $X$ has nowhere
dense closure.
\end{proof}

\begin{proposition}\label{prop:semialgebraic-rank-partition}
Let $Z$ be a semialgebraic set and let
$M:Z\longrightarrow\operatorname{Mat}_{n\times P}(\mathbb R)$ be a
semialgebraic map. For every integer $r\ge0$, the set
$Z_r:=\{z\in Z:\rank M(z)=r\}$ is semialgebraic. The nonempty sets
$Z_r$ form a finite partition of $Z$.
\end{proposition}
\begin{proof}
Note that $\rank M(z)\le r$ is equivalent to the vanishing of all
$(r+1)\times(r+1)$ minors, while $\rank M(z)\ge r$ is equivalent to the
nonvanishing of some $r\times r$ minor. These are semialgebraic conditions.
Only the ranks $0,\ldots,\min\{n,P\}$ can occur. See also
\cite[Sections~2.1--2.3]{bcr1998}.
\end{proof}

\section{Global low-dimensional coordinates}\label{sec:global-coordinates}

Assuming URP, we use the equality of dynamical dimension and mean
dimension to construct a partition of unity with uniformly controlled
local order along sufficiently invariant finite orbit windows. The resulting
orbit names lie in finite unions of products of simplex faces with the
corresponding dimension bound.

\begin{lemma}
\label{lem:global-dynamic-pou}
Assume that $G\curvearrowright X$ has URP and that
$\mdim(X,G)<a$.
Let $\mathcal W$ be a finite open cover of $X$, let
$f_0\in C(X,(0,1)^m)$, and let $\varepsilon>0$.
Then there exist an honest partition structure $\mathcal H=\hps$ 
and vectors $v_1,\ldots,v_J\in(0,1)^m$ such that
\begin{enumerate}[label=\textup{(\roman*)},leftmargin=*]
\item
 the closed cover $\mathcal D=(D_j)_{j=1}^J$ refines $\mathcal W$;

\item the affine observable
$
 f_*(x):=\sum_{j=1}^Jv_j\psi_j(x)
$
satisfies
$\|f_*-f_0\|_\infty<\varepsilon$;

\item
$
 \sup_{\mu\in  \rPG{}(X)}
 \int_X \rOrdD(x)\,d\mu(x)<a;
$

\item if $\boldsymbol{\psi}(x)=\boldsymbol{\psi}(y)$, then $x$ and $y$ lie in a common
member of $\mathcal W$.
\end{enumerate}
\end{lemma}

\begin{proof}
By \eqref{eq:dynamical-mean-dimension}, URP
implies
$ \rDdim(X,G)=\mdim(X,G)$.
Choose
$ \rDdim(X,G)<a_0<a$.

By uniform continuity of $f_0$, there is a finite open cover
$\mathcal A$ of $X$ which refines $\mathcal W$ and such that
$\|f_0(x)-f_0(y)\|_\infty<\varepsilon$
whenever $x$ and $y$ belong to a common member of
$\mathcal A$.

By the definition of $ \rDdim(X,G)$, there is a finite open refinement
$
 \mathcal B=(B_j)_{j=1}^J
$
of $\mathcal A$ such that
\[
 \sup_{\mu\in  \rPG{}(X)}
 \int_X
 \ord_{\mathcal B}(x)
 d\mu(x)
 <a_0.
\]
Discarding empty members, we may assume that every $B_j$ is
nonempty.  For each $j$, choose $A_j\in\mathcal A$ such that
$
 B_j\subset A_j.
$

Choose a closed shrinking $(C_j)_{j=1}^J$ of
$(B_j)_{j=1}^J$ such that
\[
 C_j\subset B_j
 \qquad\text{and}\qquad
 X=\bigcup_{j=1}^J C_j.
\]
Choose open sets $\widetilde{O_j}$ satisfying
$C_j\subset \widetilde{O_j}\Subset B_j$.
The sets $\widetilde{O_j}$ still cover $X$.  Set
$D_j:=\overline{\widetilde{O_j}}$.
Then $(D_j)_{j=1}^J$ is a finite closed cover of $X$, and
$D_j\subset B_j\subset A_j$.
Since $\mathcal A$ refines $\mathcal W$, each $D_j$ is contained
in a member of $\mathcal W$. So (i) holds.

For each $j$, choose a continuous function
$h_j:X\to[0,\infty)$ such that
$h_j(x)>0\Longleftrightarrow x\in \widetilde{O_j}$.
Since the $\widetilde{O_j}$'s cover $X$,
\[
 H(x):=\sum_{j=1}^Jh_j(x)>0
 \qquad(x\in X).
\]
Define
$\psi_j(x):=\frac{h_j(x)}{H(x)}$.
Then
$
 \boldsymbol{\psi}=(\psi_1,\ldots,\psi_J):X\to\Delta^{J-1}
$
is an honest partition of unity for
$(\widetilde{O_j})_{j=1}^J$, and
$\operatorname{supp}\psi_j
 = \overline{\widetilde{O_j}} = D_j$.
Thus, $\mathcal H=\hps$ is an honest partition structure.

Recall that
$
 \rOrdD(x)
 =
 \sum_{j=1}^J\mathbf 1_{D_j}(x)-1.
$
Since $D_j\subset B_j$ for every $j$,
\[
 \rOrdD(x)
 \le
 \sum_{j=1}^J\mathbf 1_{B_j}(x)-1.
\]
Consequently,
\[
 \sup_{\mu\in  \rPG{}(X)}
 \int_X\rOrdD(x)\,d\mu(x)
 \le
 \sup_{\mu\in  \rPG{}(X)}
 \int_X
 \ord_{\mathcal B}(x)
 d\mu(x)
 <a_0<a.
\]
This proves (iii).

For each $j$, choose $z_j\in B_j$ and denote
$v_j:=f_0(z_j)\in(0,1)^m$.
We always have $z_j\in B_j\subset A_j$. If $\psi_j(x)>0$, then
$x\in \widetilde{O_j}\subset B_j\subset A_j$.
Hence
$\|v_j-f_0(x)\| = \|f_0(z_j)-f_0(x)\| <\varepsilon$.
Since $\psi_j(x)\ge0$ and
$\sum_j\psi_j(x)=1$,
 \begin{align*}
 \|f_*(x)-f_0(x)\|
 =
 \left\|
 \sum_{j=1}^J
 \psi_j(x)\bigl(v_j-f_0(x)\bigr)
 \right\| \le
 \sum_{j=1}^J
 \psi_j(x)\|v_j-f_0(x)\|
 <\varepsilon.
 \end{align*}
Thus
$\|f_*-f_0\|_\infty<\varepsilon$, and (ii) holds.

Finally, suppose that
$
 \boldsymbol{\psi}(x)=\boldsymbol{\psi}(y).
$
Choose $j$ such that $\psi_j(x)>0$.  Then
$\psi_j(y)>0$ as well.  Since the partition of unity is honest,
$x,y\in \widetilde{O_j}\subset B_j\subset A_j$.
As $\mathcal A$ refines $\mathcal W$, the points $x$ and $y$
lie in a common member of $\mathcal W$.
\end{proof}

The following lemma converts a bound over invariant measures into a uniform
bound along sufficiently invariant finite sets. We include a proof, using
the empirical-measure argument related to \cite[Proposition~9.1]{meyerovitch2026}.

\begin{lemma}
\label{lem:uniform-orbit-average}
Let $w:X\to\mathbb R$ be a bounded upper semicontinuous function and
let $b\in\mathbb R$ satisfy
$
 \sup_{\mu\in  \rPG{}(X)}\int_Xw\,d\mu<b.
$
Then there exist a finite set $K\subset G$ and
$\delta>0$ such that every nonempty finite
$(K,\delta)$-invariant set $T\subset G$ satisfies
\[
 \sup_{x\in X}
 \frac1{|T|}
 \sum_{t\in T}w(tx)
 <b.
\]
\end{lemma}
\begin{proof}
Suppose, towards a contradiction, that the conclusion fails.  Choose an
increasing sequence of finite sets $K_n\subset G$ with
$e\in K_n$ and $\bigcup_nK_n=G$, and choose $\delta_n\downarrow0$.
For every $n$, there are a nonempty finite
$(K_n,\delta_n)$-invariant set $T_n\subset G$ and a point
$x_n\in X$ such that
\[
 \frac1{|T_n|}\sum_{t\in T_n}w(tx_n)\ge b.
\]
Define
\[
 \mu_n:=\frac1{|T_n|}\sum_{t\in T_n}\delta_{tx_n}\in\Prob(X).
\]
After passing to a subsequence, assume that $\mu_n\to\mu$ in the weak* topology.  Then $\mu\in \mathcal P_G(X)$.  Since $w$ is bounded and upper
semicontinuous, the Portmanteau theorem (see e.g. \cite[Theorem 2.1]{Billingsleybook1999}) gives
\[
 b
 \le
 \limsup_{n\to\infty}\int_Xw\,d\mu_n
 \le
 \int_Xw\,d\mu
 \le
 \sup_{\nu\in \mathcal P_G(X)}\int_Xw\,d\nu,
\]
contradicting the hypothesis.
\end{proof}

We now apply the preceding lemma to the function
$\rOrdD$.

\begin{corollary}
\label{cor:pou-window-control}
For the honest partition structure $\mathcal H=\hps$ obtained in
Lemma~\ref{lem:global-dynamic-pou}, there exist a finite set
$K_{\mathcal H}\subset G$, with $e\in K_{\mathcal H}$, and
$\delta_{\mathcal H}\in(0,1)$ such that every nonempty finite
$(K_{\mathcal H},\delta_{\mathcal H})$-invariant set $T\subset G$ satisfies
\[
 \sup_{x\in X}
 \sum_{t\in T}\rOrdD(tx)
 <a|T|.
\]
Moreover, for every such $T$, let
\[
 \rLocT{}
 :=
 \left\{
   \bigl(\rLocD{}(tx)\bigr)_{t\in T}:x\in X
 \right\},
\]
and define
\[
 L_T
 :=
 \bigcup_{(I_t)_{t\in T}\in\rLocT{}}
 \prod_{t\in T}\Delta_{I_t}.
\]
Then
\[
 \bfpsi^T(X)\subset L_T,
 \qquad
 \dim L_T<a|T|.
\]
\end{corollary}

\begin{proof}
The function
$
 \rOrdD(x)
 =
 \sum_{j=1}^J\mathbf 1_{D_j}(x)-1
$
is bounded and upper semicontinuous, since every $D_j$ is closed.
Lemma~\ref{lem:global-dynamic-pou} gives
\[
 \sup_{\mu\in  \rPG{}(X)}
 \int_X\rOrdD(x)\,d\mu(x)<a.
\]
Applying Lemma~\ref{lem:uniform-orbit-average} with
$w=\rOrdD$ and $b=a$, we obtain a finite set $K_{\mathcal H}\subset G$
and $\delta_{\mathcal H}>0$ such that for every nonempty finite
$(K_{\mathcal H},\delta_{\mathcal H})$-invariant set $T$,
\[
 \sup_{x\in X}
 \frac1{|T|}
 \sum_{t\in T}\rOrdD(tx)
 <a.
\]
Enlarging $K_{\mathcal H}$ and decreasing $\delta_{\mathcal H}$, if necessary, we
may assume that
$
 e\in K_{\mathcal H},
 $ and $
 \delta_{\mathcal H}\in(0,1).
$
This proves the first assertion.

Fix such a set $T$.  For $x\in X$ and $t\in T$, denote
$$I_t(x):=\rLocD{}(tx)=\{j:tx\in D_j\}.$$
Since
$
 \operatorname{supp}\psi_j\subset D_j,
$
we have
$
 \boldsymbol{\psi}(tx)\in\Delta_{I_t(x)}.
$
Therefore
\[
 \bfpsi^T(x)
 \in
 \prod_{t\in T}\Delta_{I_t(x)},
\]
and hence
$\bfpsi^T(X)\subset L_T$.
Since $\dim\Delta_{I_t(x)}=\rOrdD(tx)$, the dimension of the
corresponding product face is $\sum_{t\in T}\rOrdD(tx)$.

For every pattern
$(I_t)_{t\in T}\in\rLocT{}$, choose $x\in X$ which realizes
it.  Since $\prod_{t\in T}\Delta_{I_t}$ is a finite product of simplices,
its dimension is the sum of the dimensions of its factors.
Together with $\dim\Delta_{I_t} = \rOrdD(tx)$,
we have
\[
 \dim\left(
   \prod_{t\in T}\Delta_{I_t}
 \right)
 =
 \sum_{t\in T}\dim\Delta_{I_t} =
 \sum_{t\in T}\rOrdD(tx) <
 a|T|.
\]
Each product
$\prod_{t\in T}\Delta_{I_t}$ is a face of the finite polytope
$(\Delta^{J-1})^T$.  Hence $L_T$ is a finite union of faces of
this polytope and is contained in the finite subcomplex generated by
those faces.  Consequently,
\[
 \dim L_T
 =
 \max_{(I_t)\in\rLocT{}}
 \dim\left(
   \prod_{t\in T}\Delta_{I_t}
 \right)
 <
 a|T|.\qedhere
\]
\end{proof}
Thus the dynamical dimension hypothesis has been converted into a
uniform dimension bound on sufficiently invariant finite orbit windows.

\section{Recognition and generic embeddings}
\label{sec:global-approximation}
By the standard Baire category argument in order to prove Theorem \ref{thm:main}, it is enough to prove the following proposition.
\begin{proposition}\label{prop:eta-approximation}
Let $G$ be a countably infinite discrete amenable group acting freely
by homeomorphisms on a nonempty compact metrizable space $X$.
Suppose that the action has URPC and that $\mdim(X,G)<m/2$.
Then, for every $f_{\mathrm{orig}}\in C(X,[0,1]^m)$ and every
$\varepsilon,\eta>0$, there exists $q\in C(X,[0,1]^m)$ such that
\[
 \|q-f_{\mathrm{orig}}\|_\infty<\varepsilon
 \qquad\text{and}\qquad
 q^G \text{ is an }\eta\text{-embedding}.
\]
\end{proposition}
\begin{proof}[Proof of Theorem~\ref{thm:main} given Proposition \ref{prop:eta-approximation}]
For $\eta>0$, set
\[
 E_\eta
 :=
 \left\{
 f\in C(X,[0,1]^m):
 f^G(x)\ne f^G(y)
 \text{ whenever }\rho(x,y)\ge\eta
 \right\}.
\]
The set $E_\eta$ is open in the uniform topology.  Indeed, suppose
that $f_n\notin E_\eta$ and $f_n\to f$ uniformly.  Choose
$x_n,y_n\in X$ such that
\[
 \rho(x_n,y_n)\ge\eta
 \qquad\text{and}\qquad
 f_n^G(x_n)=f_n^G(y_n).
\]
After passing to a subsequence, we may assume that
\[
 x_n\to x
 \qquad\text{and}\qquad
 y_n\to y.
\]
Then $\rho(x,y)\ge\eta$.  For every $g\in G$,
\[
 f(gx)
 =
 \lim_{n\to\infty}f_n(gx_n)
 =
 \lim_{n\to\infty}f_n(gy_n)
 =
 f(gy),
\]
so $f^G(x)=f^G(y)$.  Hence $f\notin E_\eta$.

By Proposition~\ref{prop:eta-approximation}, $E_\eta$ is dense for
every $\eta>0$.

Finally,
\[
 \mathcal E_m(X,G)
 =
 \bigcap_{n=1}^\infty E_{1/n}.
\]
Since $C(X,[0,1]^m)$ is complete in the uniform metric, the Baire
category theorem implies that $\mathcal E_m(X,G)$ is a dense
$G_\delta$ subset of $C(X,[0,1]^m)$.
\end{proof}

It remains to prove
Proposition~\ref{prop:eta-approximation}. Our aim is to perturb the
given observable so that points with equal full orbit names\footnote{We call a pair of distinct points
with equal full orbit names a \textbf{collision}.} must
be less than $\eta$ apart. The construction has two main steps.
First, we arrange that equality of full orbit names forces the
two points into a common tower level or a common open set covering
part of the remainder. To achieve this, we state a perturbation
result which, for suitable compact--open pairs $K\subset O$,
ensures that a point in $K$ cannot have the same finite orbit name
as a point outside $O$. We then use URPC to construct the pairs
needed for the preceding conclusion.

Second, we perturb the observable further so that, within each
of these common sets, equality of full orbit names forces the
points to be less than $\eta$ apart. For this step, we construct
covers on which the relevant finite orbit maps have arbitrarily
small oscillation, with order controlled by the mean dimension
bound. These covers allow us to approximate the orbit maps on
tower bases by maps that distinguish the required small sets,
including the pieces of the remainder moved into the bases by
comparison. Cutoff functions combine these approximations into
a continuous global observable. The perturbation is small enough
to preserve the conclusion of the first step, and the two steps
together give an $\eta$-embedding. The choices of parameters and
the proof of the proposition are completed in
Subsection~\ref{subsec:parameter-assembly}.

\subsection[Recognition by a finite orbit name]{Recognition by a finite orbit name}
\label{subsec:recognition-statement}
For a finite subset $F\subset G$, denote
\[
 R_F=F^{-1}F,
 \qquad \mathcal Q_F=R_F^{-1}R_F\setminus\{e\}.
\]
\begin{proposition}
\label{prop:finite-family-localizer}
Fix $a<m/2$, an honest partition structure $\mathcal H=\hps$, vectors
$v_1,\ldots,v_J\in(0,1)^m$, and the affine observable
\[
 f_*(x):=\sum_{j=1}^Jv_j\psi_j(x).
\]
Let $A\subset G$ be finite, with $e\in A$, and assume that
\begin{equation}\label{eq:localizer-A-order}
 \sum_{c\in A}\rOrdD{}(cgx)<a|A|
 \qquad(g\in G,\ x\in X).
\end{equation}
Then there exists a finite set $F\subset G$, with $e\in F$, having
 the following property.

Let $U\subset X$ be open. Assume that the sets $tU$, $t\in F$,
are pairwise disjoint and that
\begin{equation}\label{eq:localizer-marker-separation}
 U\cap qU=\emptyset
 \qquad q\in\mathcal Q_F.
\end{equation}
Let $\Lambda\ne\emptyset$ be finite, and suppose that, for every
$\ell\in\Lambda$, $K_\ell\subset X$ is compact and
$O_\ell\subset X$ is open, with
$
 K_\ell\subset O_\ell\Subset  U,
$
and that the closures $\overline{O_\ell}$, $\ell\in\Lambda$, are
pairwise disjoint. Then, for every $\varepsilon>0$, there exists
$q\in C(X,[0,1]^m)$ such that
$\|q-f_*\|_\infty<\varepsilon$
and
\begin{equation}\label{eq:localizer-separation}
 q^F(K_\ell)\cap q^F(X\setminus O_\ell)=\emptyset
 \qquad(\ell\in\Lambda).
\end{equation}
\end{proposition}
 An observable supplied by Proposition~\ref{prop:finite-family-localizer}
is called a \textbf{(finite-window) localizer}. Equation \eqref{eq:localizer-separation} is referred to as the  \textbf{(finite-window) recognition property}. Proposition~\ref{prop:finite-family-localizer} is proved in
Section~\ref{sec:recognition-proof}, with the final construction
given in Subsection~\ref{subsec:recognition-construction}.
\begin{remark}\label{rem:prepared-localizer}
 The localizer $q_0$ in  Proposition~\ref{prop:finite-family-localizer} is of the following form. There are continuous functions
\[
 \rho_\ell:X\to[0,1],
 \qquad \ell\in\Lambda,
\]
with
\[
 \rho_\ell=1\text{ on }K_\ell,
 \qquad
 \supp\rho_\ell\subset O_\ell,
\]
as well as a continuous map
\[
 Q:(\Delta^{J-1})^A
   \times[0,1]^{F^{-1}\times\Lambda}
   \longrightarrow\mathbb R^m
\]
such that
\[
 q_0(z)
 =
 Q\bigl(
   \boldsymbol{\psi}^{A}(z),
   \boldsymbol{\rho}^{F^{-1}}(z)
 \bigr),
 \qquad z\in X,
\]
where
\[
 \boldsymbol{\rho}(z)
 :=
 (\rho_\ell(z))_{\ell\in\Lambda}.
\]
The support properties of these functions and the above factorization
will be used in Lemma~\ref{lem:scale-uniform-coordinates}.
\end{remark}

\subsection{Placing collisions in a common tower level or a remainder neighbourhood}

We follow Naryshkin's type-semigroup notation
\cite[Definition~3.7]{naryshkin2024}, a variant of Ma's
generalized type semigroup \cite{MaGeneralizedTypeSemigroup}
in which no further quotient by mutual subequivalence is taken.

Let $V_1,\ldots,V_r,W_1,\ldots,W_s\subset X$ be open sets, and let
$n_1,\ldots,n_r,m_1,\ldots,m_s\in\mathbb N_0$.
In the formal sums
$$
 \sum_{i=1}^r n_i[V_i]
 \qquad\text{and}\qquad
 \sum_{j=1}^s m_j[W_j],
$$
the coefficients specify the numbers of labelled copies of the
corresponding open sets. We call $(i,a)$, where
$1\le i\le r$ and $1\le a\le n_i$, a \textbf{source label},
and $(j,b)$, where $1\le j\le s$ and $1\le b\le m_j$,
a \textbf{target label}.

We write
$$
 \sum_{i=1}^r n_i[V_i]
 \prec
 \sum_{j=1}^s m_j[W_j]
$$
if, for every choice of compact sets $K_{i,a}\subset V_i$,
one for each source label $(i,a)$, there exist finite families
$\mathcal U_{i,a}$ of open subsets of $X$ satisfying
$$
 K_{i,a}\subset\bigcup_{U\in\mathcal U_{i,a}}U,
$$
together with, for each $U\in\mathcal U_{i,a}$, an element
$g_{i,a,U}\in G$ and a target label $\tau(i,a,U)$, such that:
\begin{enumerate}[label=\textup{(\roman*)},leftmargin=*]
\item
if $\tau(i,a,U)=(j,b)$, then $g_{i,a,U}U\subset W_j$;
\item
for each fixed target label $(j,b)$, the translated sets
$g_{i,a,U}U$, as $(i,a,U)$ ranges over all triples satisfying
$\tau(i,a,U)=(j,b)$, are pairwise disjoint.
\end{enumerate}

We shall use without further comment that this relation is
additive and transitive, and that
$$
 n_i\le m_i\ \text{for all }i
 \quad\Longrightarrow\quad
 \sum_i n_i[V_i]\prec\sum_i m_i[V_i].
$$
\begin{lemma}\cite[Lemma~3.19]{naryshkin2024}
\label{lem:marker-comparison}
Let $G$ be a countably infinite discrete amenable group acting freely
by homeomorphisms on a nonempty compact metrizable space $X$.
Let $U\subset X$ be a sweep-out set.  Then there exist
$
 \theta_U>0
$  and a finite set 
 $K_U\subset G
$
such that  if
$
 \{(S_i,V_i)\}_{i=1}^r
$
is an open castle and every shape $S_i$ is
$(K_U,\theta_U)$-invariant, then
\[
 \sum_{i=1}^r
 \left\lceil\theta_U|S_i|\right\rceil [V_i]
 \prec [U].
\]
\end{lemma}

The next proposition is extracted from and adapted to the present finite-window setting from the mechanism in Naryshkin's proof of the shift-embedding theorem \cite[Subsection 6.1]{naryshkin2024}.
\begin{proposition}
\label{lem:urpc-buffer-router}
Let $U\subset X$ be a sweep-out set, and let $\theta_U>0$
and the finite set $K_U\subset G$ be supplied by
Lemma~\ref{lem:marker-comparison}.
Fix finite sets $F,K\subset G$, a number $\alpha>0$, and a
$(K,\alpha)$-URPC witness
$$
 \mathfrak W^0
 =
 \left(
  \{S_i\}_{i=1}^r,
  \{V_i^0\}_{i=1}^r,
  \{g_j\}_{j=1}^N,
  \{U_j^0\}_{j=1}^N
 \right).
$$
Here $S_i$ and $V_i^0$ are the tower shapes and bases,
$g_j$ are the transporters, and $U_j^0$ are the open sets
covering the remainder. Let
$$
 J_{i,c},
 \qquad
 1\le i\le r,\quad
 1\le c\le\lfloor\alpha|S_i|\rfloor,
$$
be the partition of $\{1,\ldots,N\}$ associated with this witness.
For each $j$, write $i(j)$ and $c(j)$ for the unique indices
such that $j\in J_{i(j),c(j)}$.

Assume that every $S_i$ is also $(K_U,\theta_U)$-invariant and that
\begin{equation}\label{eq:routing-copy-budget}
 1+\lfloor\alpha|S_i|\rfloor
 \le \left\lceil\theta_U|S_i|\right\rceil
 \qquad(1\le i\le r).
\end{equation}
Then there exist open sets
\footnote{The sets $B_i$ will be used to construct the cutoff
functions in Subsection~\ref{subsec:global-perturbation}.}
\begin{equation}\label{eq:routing-nested-sets}
 V_i^-\Subset V_i^+\Subset B_i\Subset V_i^0
 \qquad(1\le i\le r),
\end{equation}
and
$$
 U_j^-\Subset U_j^+\Subset U_j^0
 \qquad(1\le j\le N),
$$
such that
$$
 \mathfrak W^\pm
 :=
 \left(
  \{S_i\}_{i=1}^r,
  \{V_i^\pm\}_{i=1}^r,
  \{g_j\}_{j=1}^N,
  \{U_j^\pm\}_{j=1}^N
 \right)
$$
are both $(K,\alpha)$-URPC witnesses with the same shapes,
transporters, and index partition as $\mathfrak W^0$.
Moreover:
\begin{enumerate}[label=\textup{(\roman*)},leftmargin=*]
\item
For every $1\le j\le N$,
\begin{equation}\label{eq:routing-transport-buffer}
 g_j\overline{U_j^+}\subset V_{i(j)}^+.
\end{equation}

\item
There exist a nonempty finite index set $\Lambda$, compact sets
$K_\ell\subset X$, and open sets $O_\ell\subset X$, such that
\begin{equation}\label{eq:routing-localizer-family}
 K_\ell\subset O_\ell\Subset U
 \qquad(\ell\in\Lambda),
\end{equation}
and the closures $\overline{O_\ell}$ are pairwise disjoint.
If $q\in C(X,[0,1]^m)$ satisfies the finite-window recognition
property
\begin{equation}\label{eq:routing-localizer-condition}
 q^F(K_\ell)\cap q^F(X\setminus O_\ell)=\emptyset
 \qquad(\ell\in\Lambda),
\end{equation}
then any two points $x,y\in X$ with $q^G(x)=q^G(y)$ belong
to a common member of the finite open cover
\begin{equation}\label{eq:routing-plus-cover}
 \mathcal C^+
 :=
 \{sV_i^+:1\le i\le r,\ s\in S_i\}
 \cup
 \{U_j^+:1\le j\le N\}.
\end{equation}
\end{enumerate}
\end{proposition}

Each pair $(K_\ell,O_\ell)$ in
Proposition~\ref{lem:urpc-buffer-router} is called a
\textbf{routing pair}, and $O_\ell$ is its open \textbf{buffer}.
We call the witnesses $\mathfrak W^-$ and $\mathfrak W^+$,
the intermediate bases $B_i$, and the routing pairs the
\textbf{buffered URPC data} associated with $\mathfrak W^0$.

\begin{proof}
We first construct the nested witnesses and the intermediate
bases. Since $\mathfrak W^0$ is a URPC witness, its tower levels
and remainder sets form a finite open cover of $X$.
Choose compact sets
$$
 E_{i,s}\subset sV_i^0,
 \qquad
 H_j\subset U_j^0,
$$
whose union is $X$. For each $j$, choose an open set $U_j^+$
such that
$$
 H_j\subset U_j^+\Subset U_j^0.
$$
For each $i$, the compact set
$$
 \bigcup_{s\in S_i}s^{-1}E_{i,s}
 \;\cup\!
 \bigcup_{\{j:i(j)=i\}}g_j\overline{U_j^+}
$$
is contained in $V_i^0$. Choose an open neighbourhood $V_i^+$
of this compact set with $V_i^+\Subset V_i^0$.

The sets $sV_i^+$ and $U_j^+$ still cover $X$, since they
contain $E_{i,s}$ and $H_j$, respectively. The castle
disjointness is inherited from $\mathfrak W^0$, and
$$
 g_j\overline{U_j^+}\subset V_{i(j)}^+.
$$
For each fixed $(i,c)$, the transported closures
$g_j\overline{U_j^+}$, $j\in J_{i,c}$, are pairwise disjoint,
since they lie in the pairwise disjoint sets $g_jU_j^0$.
Thus $\mathfrak W^+$ is a $(K,\alpha)$-URPC witness with the
original shapes, transporters, and index partition.

Apply the same construction to $\mathfrak W^+$ to obtain
$\mathfrak W^-$, with
$$
 V_i^-\Subset V_i^+,
 \qquad
 U_j^-\Subset U_j^+.
$$
Finally, choose open sets $B_i$ satisfying
$$
 V_i^+\Subset B_i\Subset V_i^0.
$$

We next use comparison to move pieces of all the bases $V_i^+$
and remainder sets $U_j^+$ into $U$, with disjoint images.
For each fixed $(i,c)$, the sets $g_jU_j^+$,
$j\in J_{i,c}$, are pairwise disjoint subsets of $V_i^+$.
Therefore
\begin{equation}\label{eq:routing-remainder-comparison}
 \sum_{j=1}^N[U_j^+]
 \prec
 \sum_{i=1}^r
 \lfloor\alpha|S_i|\rfloor[V_i^+].
\end{equation}
By additivity, \eqref{eq:routing-copy-budget}, and
Lemma~\ref{lem:marker-comparison},
$$
\begin{aligned}
 \sum_{i=1}^r[V_i^+]+\sum_{j=1}^N[U_j^+]
 &\prec
 \sum_{i=1}^r
 \bigl(1+\lfloor\alpha|S_i|\rfloor\bigr)[V_i^+]\\
 &\prec
 \sum_{i=1}^r
 \left\lceil\theta_U|S_i|\right\rceil[V_i^+]
 \prec [U].
\end{aligned}
$$
The last comparison applies because
$\{(S_i,V_i^+)\}_{i=1}^r$ is an open castle with
$(K_U,\theta_U)$-invariant shapes. In particular,
\begin{equation}\label{eq:routing-total-comparison}
 \sum_{i=1}^r[V_i^+]+\sum_{j=1}^N[U_j^+]\prec[U].
\end{equation}

Since $\mathfrak W^-$ is a URPC witness, choose compact sets
$$
 C_{i,s}\subset sV_i^-,
 \qquad
 D_j\subset U_j^-,
$$
such that
\begin{equation}\label{eq:routing-closed-shrinking}
 X=
 \bigcup_{i=1}^r\bigcup_{s\in S_i}C_{i,s}
 \cup
 \bigcup_{j=1}^ND_j.
\end{equation}
For each $i$, put
$$
 C_i:=\bigcup_{s\in S_i}s^{-1}C_{i,s}.
$$
Then $C_i$ is compact and contained in $V_i^-$.

Apply \eqref{eq:routing-total-comparison} to the compact source
sets $C_i\subset V_i^+$ and $D_j\subset U_j^+$.
Intersect the implementing open sets with their respective
source sets $V_i^+$ or $U_j^+$. By compactness, shrink these
open sets while retaining covers of $C_i$ and $D_j$, and then
choose compact shrinkings of those covers.
This gives integers $L_i,M_j\ge0$, compact sets
$A_{i,a}^V,A_{j,b}^U$, open sets $P_{i,a}^V,P_{j,b}^U$,
and elements $r_{i,a}^V,r_{j,b}^U\in G$, satisfying
$$
\begin{gathered}
 C_i\subset\bigcup_{a=1}^{L_i}A_{i,a}^V,
 \qquad
 A_{i,a}^V\subset P_{i,a}^V\Subset V_i^+,
 \qquad
 1\le i\le r,\quad 1\le a\le L_i,\\
 D_j\subset\bigcup_{b=1}^{M_j}A_{j,b}^U,
 \qquad
 A_{j,b}^U\subset P_{j,b}^U\Subset U_j^+,
 \qquad
 1\le j\le N,\quad 1\le b\le M_j,
\end{gathered}
$$
such that all the sets
$$
 r_{i,a}^V\overline{P_{i,a}^V}
 \quad\text{and}\quad
 r_{j,b}^U\overline{P_{j,b}^U}
$$
form one pairwise disjoint family of compact subsets of $U$.

Define
$$
\begin{aligned}
 \Lambda
 &:=
 \{(V,i,a):1\le i\le r,\ 1\le a\le L_i\}\\
 &\qquad\sqcup
 \{(U,j,b):1\le j\le N,\ 1\le b\le M_j\}.
\end{aligned}
$$
For $\ell=(V,i,a)$, set
$$
 K_\ell:=r_{i,a}^VA_{i,a}^V,
 \qquad
 O_\ell:=r_{i,a}^VP_{i,a}^V,
$$
and for $\ell=(U,j,b)$, set
$$
 K_\ell:=r_{j,b}^UA_{j,b}^U,
 \qquad
 O_\ell:=r_{j,b}^UP_{j,b}^U.
$$
These pairs satisfy $K_\ell\subset O_\ell\Subset U$, and the
closures $\overline{O_\ell}$ are pairwise disjoint.
Since the sets $C_{i,s}$ and $D_j$ cover the nonempty space $X$,
at least one $C_i$ or $D_j$ is nonempty. Hence at least one
$L_i$ or $M_j$ is positive, so $\Lambda\ne\emptyset$.

It remains to verify the conclusion concerning equal orbit
names. Suppose that $q$ satisfies the finite-window recognition
property~\eqref{eq:routing-localizer-condition} and that
$q^G(x)=q^G(y)$. Then
$$
 q^G(hx)=q^G(hy)\qquad(h\in G),
$$
because $q(ghx)=q(ghy)$ for every $g,h\in G$.

If $x\in C_{i,s}$, put $u=s^{-1}x$ and $v=s^{-1}y$.
Choose $a$ such that $u\in A_{i,a}^V$, and let
$\ell=(V,i,a)$. Then
$$
 r_{i,a}^Vu\in K_\ell,
 \qquad
 q^F(r_{i,a}^Vu)=q^F(r_{i,a}^Vv).
$$
The finite-window recognition
property~\eqref{eq:routing-localizer-condition} implies
$r_{i,a}^Vv\in O_\ell$, and hence $v\in P_{i,a}^V$.
Since
$$
 u\in A_{i,a}^V\subset P_{i,a}^V\subset V_i^+,
$$
we obtain $u,v\in V_i^+$, and therefore $x,y\in sV_i^+$.

If $x\in D_j$, choose $b$ such that $x\in A_{j,b}^U$,
and let $\ell=(U,j,b)$. Then
$$
 r_{j,b}^Ux\in K_\ell,
 \qquad
 q^F(r_{j,b}^Ux)=q^F(r_{j,b}^Uy).
$$
The finite-window recognition
property~\eqref{eq:routing-localizer-condition} gives
$r_{j,b}^Uy\in O_\ell$, so $y\in P_{j,b}^U$.
Together with
$$
 x\in A_{j,b}^U\subset P_{j,b}^U\subset U_j^+,
$$
this yields $x,y\in U_j^+$.

The compact sets $C_{i,s}$ and $D_j$ cover $X$ by
\eqref{eq:routing-closed-shrinking}, so these cases exhaust all
possibilities. Thus $x$ and $y$ belong to a common member
of $\mathcal C^+$.
\end{proof}

\subsection[Covers of controlled order and small oscillation]{Covers of controlled order and small oscillation}

Fix an observable $q_0$ with the representation recorded in Remark \ref{rem:prepared-localizer}. For a sufficiently invariant
finite set $S$, we will construct covers on which $q_0^S$ has arbitrarily small
oscillation, while their order stays bounded linearly in $|S|$.

\begin{lemma}\label{lem:marker-sparsity}
Let $F,R\subset G$ be finite, with $R\ne\emptyset$, and let
$U\subset X$ be open.  Suppose that the sets $rU$, $r\in R$, are
pairwise disjoint.  Let $\kappa,\varepsilon_R>0$ satisfy
$(1+\varepsilon_R)/|R|<\kappa$.  If a nonempty finite set
$S\subset G$ satisfies $|RF^{-1}S|<(1+\varepsilon_R)|S|$, then,
for every $x\in X$,
\begin{equation}\label{eq:marker-sparse-visits}
 \bigl|\{u\in F^{-1}S:ux\in U\}\bigr|<\kappa|S|.
\end{equation}
\end{lemma}

\begin{proof}
Fix $x\in X$ and put
$$
 T_x:=\{u\in F^{-1}S:ux\in U\}.
$$
We claim that the multiplication map
$$
 R\times T_x\longrightarrow RF^{-1}S,
 \qquad (r,u)\longmapsto ru,
$$
is injective. Indeed, suppose that
$
 r_1u_1=r_2u_2
$
for $r_1,r_2\in R$ and $u_1,u_2\in T_x$.  Since
$u_1x,u_2x\in U$, we have
$r_1u_1x=r_2u_2x\in r_1U\cap r_2U$.
The translates $rU$, $r\in R$, are pairwise disjoint, so
$r_1=r_2$.  As $r_1u_1=r_2u_2$, we conclude $u_1=u_2$.

Hence
$|R||T_x| \le |RF^{-1}S|$.
Using the two hypotheses above, we obtain
$|T_x| < \frac{1+\varepsilon_R}{|R|}|S| < \kappa|S|$.
This proves \eqref{eq:marker-sparse-visits}.
\end{proof}

Fix $a$ with $\mdim(X,G)<a<m/2$ and a finite open cover
$\mathcal W$ of $X$. Let
$$
 \mathcal H=\hps,
 \qquad
 v_1,\ldots,v_J\in(0,1)^m,
$$
be supplied by Lemma~\ref{lem:global-dynamic-pou}, and put
$$
 f_*(z):=\sum_{j=1}^Jv_j\psi_j(z).
$$
Let $K_{\mathcal H}\subset G$ and
$\delta_{\mathcal H}\in(0,1)$ be the corresponding finite-invariance
data from Corollary~\ref{cor:pou-window-control}.

Fix a finite set $A\subset G$ containing $e$ and satisfying
\eqref{eq:localizer-A-order}, and let $F\subset G$ be supplied by
Proposition~\ref{prop:finite-family-localizer}.
Let $U\subset X$ be an open sweep-out set, and let
$(K_\ell,O_\ell)_{\ell\in\Lambda}$ be a nonempty finite family
satisfying the hypotheses of that proposition.
Fix a corresponding localizer $q_0$, together with the continuous
functions $\rho_\ell:X\to[0,1]$ from
Remark~\ref{rem:prepared-localizer}. Recall that
$$
 \rho_\ell=1\text{ on }K_\ell,
 \qquad
 \supp\rho_\ell\subset O_\ell
 \qquad(\ell\in\Lambda).
$$
Write
$$
 \boldsymbol{\rho}:X\longrightarrow[0,1]^\Lambda,
 \qquad
 \boldsymbol{\rho}(z):=(\rho_\ell(z))_{\ell\in\Lambda},
$$
and, for each finite set $E\subset G$, set
$$
 \boldsymbol{\rho}^{E}(z)
 :=
 \bigl(\rho_\ell(uz)\bigr)_{(u,\ell)\in E\times\Lambda}.
$$
\begin{lemma}
\label{lem:scale-uniform-coordinates}
With the data fixed above, let
$\varepsilon_A,\kappa,\varepsilon_R>0$, and let $R\subset G$
be a nonempty finite set such that the translates $rU$,
$r\in R$, are pairwise disjoint and
$$
 \frac{1+\varepsilon_R}{|R|}<\kappa.
$$
Let $S\subset G$ be a nonempty finite set such that $AS$ is
$(K_{\mathcal H},\delta_{\mathcal H})$-invariant and
$$
 |AS|<(1+\varepsilon_A)|S|,
 \qquad
 |RF^{-1}S|<(1+\varepsilon_R)|S|.
$$
Define
\begin{equation}\label{eq:coordinate-map}
\begin{aligned}
 \Pi_S:X&\longrightarrow
 (\Delta^{J-1})^{AS}
 \times[0,1]^{F^{-1}S\times\Lambda},\\
 \Pi_S(z)&:=
 \bigl(
   \boldsymbol{\psi}^{AS}(z),
   \boldsymbol{\rho}^{F^{-1}S}(z)
 \bigr),
\end{aligned}
\end{equation}
and put $P_S:=\Pi_S(X)$. Then there exists a continuous map
$$
 \overline q_S:P_S\longrightarrow([0,1]^m)^S
$$
such that
\begin{equation}\label{eq:coordinate-recovery}
 q_0^S=\overline q_S\circ\Pi_S,
\end{equation}
and
\begin{equation}\label{eq:coordinate-dimension-bound}
 \dim P_S<
 \bigl(a(1+\varepsilon_A)+\kappa\bigr)|S|.
\end{equation}

Moreover, for every compact set $Y\subset X$ and every
$\tau>0$, there exists a finite relatively open cover
$\mathcal U$ of $Y$ such that:
\begin{enumerate}[label=\textup{(\roman*)},leftmargin=*]
\item
$\mathcal U$ refines
$$
 \left.\bigvee_{s\in S}s^{-1}\mathcal W\right|_Y;
$$
\item
the oscillation of $q_0^S$ on each member of $\mathcal U$
is less than $\tau$;
\item
$$
 \ord(\mathcal U)<
 \bigl(a(1+\varepsilon_A)+\kappa\bigr)|S|.
$$
\end{enumerate}
In particular, $P_S$, $\Pi_S$, $\overline q_S$, and the stated
dimension bound are fixed before the oscillation scale $\tau$
is chosen.
\end{lemma}

\begin{proof}
By Remark~\ref{rem:prepared-localizer}, there exists a continuous map
\[
 Q:
 (\Delta^{J-1})^A
 \times
 [0,1]^{F^{-1}\times\Lambda}
 \longrightarrow \mathbb R^m
\]
such that
\begin{equation}\label{eq:coordinate-one-site-formula}
 q_0(z)
 =
 Q\bigl(
   \boldsymbol{\psi}^{A}(z),
   \boldsymbol{\rho}^{F^{-1}}(z)
 \bigr),
 \qquad z\in X.
\end{equation}

For $s\in S$, applying
\eqref{eq:coordinate-one-site-formula} to $sz$ gives
\[
 q_0(sz)
 =
 Q\left(
   \bigl(\boldsymbol{\psi}(csz)\bigr)_{c\in A},
   \bigl(\rho_\ell(us z)\bigr)_
        {(u,\ell)\in F^{-1}\times\Lambda}
 \right).
\]
Since
$
 cs\in AS
$ and $
 us\in F^{-1}S,$
all the entries on the right-hand side are coordinates of
$\Pi_S(z)$.
More explicitly, write a point of
$
 (\Delta^{J-1})^{AS}
 \times
 [0,1]^{F^{-1}S\times\Lambda}
$
as $(\xi,\eta)$, where
$
 \xi=(\xi_v)_{v\in AS},
$ and $ \eta=(\eta_{v,\ell})_
      {(v,\ell)\in F^{-1}S\times\Lambda}.$
Define
\[
 \widetilde q_S(\xi,\eta)
 :=
 \left(
 Q\left(
   (\xi_{cs})_{c\in A},
   (\eta_{us,\ell})_
        {(u,\ell)\in F^{-1}\times\Lambda}
 \right)
 \right)_{s\in S}.
\]
Then $\widetilde q_S$ is continuous and
$
 q_0^S=\widetilde q_S\circ\Pi_S.
$
Restricting $\widetilde q_S$ to
$P_S=\Pi_S(X)$, we obtain a continuous map
$
 \overline q_S:=\widetilde q_S|_{P_S}
 :P_S\to([0,1]^m)^S
$
satisfying
$
 q_0^S=\overline q_S\circ\Pi_S.
$
This proves \eqref{eq:coordinate-recovery}.

We next estimate the dimension of $P_S$ to prove
\eqref{eq:coordinate-dimension-bound}.
For $z\in X$, set
$
 E_z:=\{u\in F^{-1}S:uz\in U\}.
$
By Lemma~\ref{lem:marker-sparsity},
\begin{equation}\label{eq:coordinate-sparse-roots}
 |E_z|<\kappa|S|.
\end{equation}
For $z\in X$, define the set of active localizing roots by\footnote{Since $\supp\rho_\ell\subset O_\ell\subset U$ for every
$\ell\in\Lambda$, it holds $E_z^{\mathrm{loc}}\subset E_z$.}
\[
 E_z^{\mathrm{loc}}
 :=
 \left\{
 u\in F^{-1}S:
 \rho_\ell(uz)\neq 0
 \text{ for some }\ell\in\Lambda
 \right\}\subset E_z.
\]
Therefore, by \eqref{eq:coordinate-sparse-roots},
\begin{equation}\label{eq:coordinate-active-localizing-roots}
 |E_z^{\mathrm{loc}}|<\kappa|S|.
\end{equation}

Moreover, for every $u\in E_z^{\mathrm{loc}}$, there is a unique
index $\lambda_z(u)\in\Lambda$ such that
$
 \rho_{\lambda_z(u)}(uz)\neq 0.
$
Indeed, if both $\rho_\ell(uz)$ and $\rho_{\ell'}(uz)$ were
nonzero, then
$
 uz\in O_\ell\cap O_{\ell'},
$
and the pairwise disjointness of the sets $O_\ell$ would imply
$\ell=\ell'$.
Consequently, $\boldsymbol{\rho}^{F^{-1}S}(z)$ has fewer than
$\kappa|S|$ nonzero scalar coordinates. For each fixed set of
positions in which nonzero coordinates may occur, these vectors
lie in a coordinate cube whose dimension is the number of those
positions. We now define these cubes.

For a subset $E\subset F^{-1}S$ and a map $\lambda:E\to\Lambda$, let
\[
 \Gamma(E,\lambda)
 :=
 \{(u,\lambda(u)):u\in E\}
 \subset F^{-1}S\times\Lambda.
\]
Define
\[
 Z(E,\lambda)
 :=
 \left\{
 w\in[0,1]^{F^{-1}S\times\Lambda}:
 w_{u,\ell}=0
 \text{ whenever }(u,\ell)\notin\Gamma(E,\lambda)
 \right\}.
\]
Thus only the coordinates indexed by $\Gamma(E,\lambda)$ are allowed
to vary, and the coordinate projection
\[
 Z(E,\lambda)\longrightarrow[0,1]^E,
 \qquad
 w\longmapsto\bigl(w_{u,\lambda(u)}\bigr)_{u\in E},
\]
is a homeomorphism.  In particular,
\begin{equation}\label{eq:coordinate-coordinate-box-dimension}
 \dim Z(E,\lambda)=|E|.
\end{equation}

We now collect all support patterns allowed by
\eqref{eq:coordinate-active-localizing-roots}.  Set
\[
 Z_S
 :=
 \bigcup_{\substack{
 E\subset F^{-1}S,\ |E|<\kappa|S|\\
 \lambda:E\to\Lambda}}
 Z(E,\lambda).
\]
This is a finite union, since both $F^{-1}S$ and $\Lambda$ are finite.
For every $z\in X$, the active coordinates of
$\boldsymbol{\rho}^{F^{-1}S}(z)$ are precisely those indexed by
$\Gamma(E_z^{\mathrm{loc}},\lambda_z)$.  Hence
\[
 \boldsymbol{\rho}^{F^{-1}S}(z)
 \in Z(E_z^{\mathrm{loc}},\lambda_z)
 \subset Z_S,
\]
and therefore
\[
 \boldsymbol{\rho}^{F^{-1}S}(X)\subset Z_S.
\]

It remains to estimate the dimension of $Z_S$. Each cube $Z(E,\lambda)$ is compact, hence closed in $Z_S$,
and has dimension $|E|<\kappa|S|$.
Since $Z_S$ is a finite union of these cubes, the finite
 sum theorem for covering dimension of closed sets gives
\begin{equation}\label{eq:coordinate-localizing-dimension}
 \dim Z_S
 =
 \max_{\substack{
 E\subset F^{-1}S,\ |E|<\kappa|S|\\
 \lambda:E\to\Lambda}}
 \dim Z(E,\lambda)
 <\kappa|S|.
\end{equation}

By hypothesis, the finite set $AS$ is
$(K_{\mathcal H},\delta_{\mathcal H})$-invariant.  Hence
Corollary~\ref{cor:pou-window-control}, applied with $T=AS$, gives
a finite union $L_{AS}$ of products of simplex faces such that
\[
 \boldsymbol{\psi}^{AS}(X)\subset L_{AS}
\qquad\text{and}\qquad
 \dim L_{AS}<a|AS|.
\]
Using $|AS|<(1+\varepsilon_A)|S|$, we therefore obtain
\begin{equation}\label{eq:coordinate-pou-dimension}
 \dim L_{AS}
 <
 a|AS|
 <
 a(1+\varepsilon_A)|S|.
\end{equation}

Since
$
 \Pi_S(z)
 =
 \bigl(
   \boldsymbol{\psi}^{AS}(z),
   \boldsymbol{\rho}^{F^{-1}S}(z)
 \bigr),
$
the two inclusions above imply, for every $z\in X$,
$
 \Pi_S(z)\in L_{AS}\times Z_S.
$
Consequently,
\begin{equation*}
 P_S=\Pi_S(X)\subset L_{AS}\times Z_S.
\end{equation*}
Since $P_S\subset L_{AS}\times Z_S$, monotonicity and the
product theorem for covering dimension of metrizable spaces give
\begin{align*}
 \dim P_S
 \le \dim(L_{AS}\times Z_S)\le \dim L_{AS}+\dim Z_S<
 \bigl(a(1+\varepsilon_A)+\kappa\bigr)|S|.
 \end{align*}
This proves \eqref{eq:coordinate-dimension-bound}.

Fix a compact set $Y\subset X$, and set
$ P_Y:=\Pi_S(Y)\subset P_S.$
We first show that every fibre of
$
 \pi:=\Pi_S|_Y:Y\to P_Y
$
is contained in a single member of the restricted $S$-join
$
 \left.\bigvee_{s\in S}s^{-1}\mathcal W\right|_Y.
$
Indeed,
fix $p\in P_Y$.  Since $e\in A$, for every $s\in S$ we have
$
 s=es\in AS.
$
Hence the coordinate $\boldsymbol{\psi}(sz)$ is recorded by
$\boldsymbol{\psi}^{AS}(z)$, and therefore it is constant as $z$ ranges over
the fibre
$
 \pi^{-1}(p)
 =
 \Pi_S^{-1}(p)\cap Y.
$
Denote this common vector by
$
 \xi_s\in\Delta^{J-1}.
$
For each $s\in S$, choose $j_s\in\{1,\ldots,J\}$ such that
$
 (\xi_s)_{j_s}>0.
$
By honesty of the partition of unity, for any $z\in\pi^{-1}(p)$,
\[
 \psi_{j_s}(sz)= (\xi_s)_{j_s}>0
 \quad\Longrightarrow\quad
 sz\in\widetilde{O_{j_s}}.
\]
Thus, 
\begin{equation}\label{eq:coordinate-fibre-chart}
\pi^{-1}(p)\subset U_p :=
 Y\cap\bigcap_{s\in S}s^{-1}\widetilde{ O_{j_s}}.
\end{equation}
For each $s\in S$, choose $W_s\in\mathcal W$ such that
$
 \widetilde{O_{j_s}}\subset D_{j_s}\subset W_s.
$
Then
$
 U_p
 \subset
 Y\cap\bigcap_{s\in S}s^{-1}W_s,
$
and the set on the right is a member of
$
 \left.\bigvee_{s\in S}s^{-1}\mathcal W\right|_Y.
$
Thus each fibre of $\Pi_S|_Y$ is contained in one member of the
restricted $S$-join.

For every $p\in P_Y$, choose a member $U_p$ of the restricted
$S$-join which contains $\pi^{-1}(p)$.  Since $Y$ is compact and
$P_Y$ is Hausdorff, $\pi$ is a closed map.  Therefore
$V_p:=P_Y\setminus\pi(Y\setminus U_p)$
is an open neighbourhood of $p$ satisfying
$\pi^{-1}(V_p)\subset U_p$.  By continuity of
$\overline q_S|_{P_Y}$, choose an open neighbourhood
$W_p\subset P_Y$ of $p$ on which $\overline q_S$ has oscillation
less than $\tau$.  A finite subfamily of
$\{V_p\cap W_p:p\in P_Y\}$ covers $P_Y$.

Since
\[
 \dim P_Y\le\dim P_S
 <\bigl(a(1+\varepsilon_A)+\kappa\bigr)|S|,
\]
the definition of covering dimension yields a finite open refinement
$\mathcal R$ of this cover such that
\[
 \ord(\mathcal R)\le\dim P_Y
 <\bigl(a(1+\varepsilon_A)+\kappa\bigr)|S|.
\]
Set
$\mathcal U:=\{\pi^{-1}(R):R\in\mathcal R\}$, viewed as a relatively open cover of $Y$.  This cover refines
the restricted $S$-join, and, since
$q_0^S=\overline q_S\circ\pi$, the oscillation of $q_0^S$ on each
member is less than $\tau$.  Finally,
\[
 \ord(\mathcal U)\le\ord(\mathcal R)
 <\bigl(a(1+\varepsilon_A)+\kappa\bigr)|S|.
\]
This proves all assertions.
\end{proof}
\subsection[The global perturbation]{The global perturbation}\label{subsec:global-perturbation}

The preceding constructions play two complementary roles.  The
finite-window recognition property of $q_0$, together with the
routing pairs, ensures that any collision of a sufficiently
small perturbation of $q_0$ is routed into a common tower or
remainder chart.  On the other hand,
Lemma~\ref{lem:scale-uniform-coordinates} produces, on every sufficiently
invariant shape $S$, a low-order cover on which the block
$q_0^S$ has arbitrarily small oscillation.

On each buffered tower base we perturb $q_0^{S_i}$ to a nearby map
encoding the low-order cover supplied by
Lemma~\ref{lem:scale-uniform-coordinates} and the transported remainder
pieces. Cutoff functions supported inside the buffered bases combine
these block maps with $q_0$ into a single global observable $q$.
The support conditions ensure continuity, and the small perturbation
preserves finite-window recognition. Since the castle levels are
pairwise disjoint, at most one tower cutoff contributes at any point.

The local-to-global tower assembly here parallels the corresponding
construction in Gutman--Tsukamoto
\cite[Lemma~2.1 and the proof of Proposition~3.1]{gutmantsukamoto2014}.
There, block maps defined on clopen tower bases are placed along the
corresponding tower levels to obtain a global observable. In our setting the tower levels are only open and a remainder may be
present, so this direct piecewise assembly is unavailable. Buffered
cutoffs provide continuity, while finite-window recognition and
comparison handle the routing and the remainder.

\begin{lemma}
\label{lem:naryshkin-payload}
Let $Y$ be a compact metrizable space, let
$F:Y\to[0,1]^N$ be continuous, and let $\delta>0$.
Let $\mathcal U$ be a finite open cover of $Y$ and
$\mathcal V$ a finite family of open subsets of $Y$,
with repetitions discarded.
Assume that%
\footnote{Naryshkin's notion of order in
\cite[Definition~6.8]{naryshkin2024} is our multiplicity.}
\begin{equation}\label{eq:naryshkin-payload-dimension}
 \mult(\mathcal U)+\mult(\mathcal V)<\frac N2,
\end{equation}
and that the oscillation of $F$ on each member of $\mathcal U$
is less than $\delta/2$.
Then there exists a continuous map
$\widetilde F:Y\to[0,1]^N$ such that
$$
 \|\widetilde F-F\|_\infty<\delta
$$
and $\widetilde F$ encodes every member of
$\mathcal U\cup\mathcal V$.
\end{lemma}

\begin{proof}
Choose $\gamma\in(0,1)$ sufficiently small that
$$
 \mult(\mathcal U)+\mult(\mathcal V)
 \le (1-\gamma)N/2.
$$
Apply \cite[Lemma~6.14]{naryshkin2024} to $\mathcal U$
and $\mathcal V$. Taking the full coordinate set
$I=\{1,\ldots,N\}$, for which $|I|=N>(1-\gamma)N$,
gives the required approximation and encoding properties.
\end{proof}
\begin{proposition}
\label{prop:global-perturbation}
Fix $\eta>0$, and let $\mathcal W$ be a finite open cover of $X$ with
$\mesh(\mathcal W)<\eta$.
Retain the data associated with the prepared localizer $q_0$ from
Lemma~\ref{lem:scale-uniform-coordinates}, in particular
$F,U,R,\varepsilon_A,\kappa$, and $\varepsilon_R$.

Let $\alpha>0$, and suppose that we are given buffered URPC data as in
 Proposition~\ref{lem:urpc-buffer-router}, with shapes $S_i$, tower bases
$
V_i^-\Subset V_i^+\Subset B_i,
$
remainder sets
$
U_j^-\Subset U_j^+,
$
transporters $g_j$, and routing pairs
$$
K_\ell\subset O_\ell\Subset U,
\qquad \ell\in\Lambda.
$$
Denote
\[
N_{\mathcal W}:=\mult(\mathcal W)=\ord(\mathcal W)+1
\qquad\text{and}\qquad
d_0:=a(1+\varepsilon_A)+\kappa+N_{\mathcal W}\alpha.
\]
Suppose that the following conditions hold:
\begin{enumerate}[label=\textup{(\roman*)},leftmargin=*]
\item
$
q_0^F(K_\ell)\cap q_0^F(X\setminus O_\ell)=\emptyset
$, $\ell\in\Lambda.
$

\item  $d_0<d<m/2$ for some $d\in \mathbb{R}$.

\item Every shape $S_i$ satisfies
$
|AS_i|<(1+\varepsilon_A)|S_i|,
$
the set $AS_i$ is $(K_{\mathcal H},\delta_{\mathcal H})$-invariant,
$
|RF^{-1}S_i|<(1+\varepsilon_R)|S_i|,
$
and
$
|S_i|>\frac1{d-d_0}.
$
\end{enumerate}
Then, for every $\varepsilon>0$, there exists
$q\in C(X,[0,1]^m)$ such that
$
\|q-q_0\|_\infty<\varepsilon
$
and $q^G$ is an $\eta$-embedding.
\end{proposition}

\begin{proof}
For every $\ell\in\Lambda$, the compact sets
$
 q_0^F(K_\ell)
$ and $ q_0^F(X\setminus O_\ell)
$
are disjoint.  Since $\Lambda$ is finite, we may set
\begin{equation}\label{eq:payload-separation-margin}
 \sigma
 :=
 \min_{\ell\in\Lambda}
 \mathrm{dist}\bigl(
 q_0^F(K_\ell),q_0^F(X\setminus O_\ell)
 \bigr)>0.
\end{equation}
Choose
$
 0<\delta<\min\{\varepsilon,\sigma/3\}.
$

For each $i$,  denote $Y_i:=\overline{B_i}$.  Applying
Lemma~\ref{lem:scale-uniform-coordinates} with
$S=S_i$, $Y=Y_i$, and $\tau=\delta/2$, we obtain a finite
relatively open cover $\mathcal U_i$ of $Y_i$ which refines
$
 \left.
 \bigvee_{s\in S_i}s^{-1}\mathcal W
 \right|_{Y_i},
$
such that $q_0^{S_i}$ has oscillation less than $\delta/2$ on
every member of $\mathcal U_i$, and
\begin{equation}\label{eq:payload-U-order}
 \ord(\mathcal U_i)
 <
 \bigl(a(1+\varepsilon_A)+\kappa\bigr)|S_i|.
\end{equation}

For each base index $i$, use all remainder indices $j$ with
$i(j)=i$ and all members $W\in\mathcal W$ to form the finite family
\begin{equation}\label{eq:payload-V-family}
 \mathcal V_i:=
 \left\{Y_i\cap g_j(U_j^+\cap W):
  1\le j\le N,\ i(j)=i,\ W\in\mathcal W\right\}
 \setminus\{\emptyset\}.
\end{equation}
Repeated subsets occur only once. At a given point $y\in Y_i$,
each  color contributes at most one $j$, because the
sets $g_jU_j^+$ are disjoint within that color. For this $j$,
$g_j^{-1}y$ belongs to at most
$N_{\mathcal W}=\mult(\mathcal W)$ members of $\mathcal W$.
Thus
\[
 \mult(\mathcal V_i)
 \le N_{\mathcal W}\lfloor\alpha|S_i|\rfloor
 \le N_{\mathcal W}\alpha|S_i|.
\]
By \eqref{eq:payload-U-order} and
$\mult(\mathcal U_i)=\ord(\mathcal U_i)+1$,
\[
\begin{aligned}
 \mult(\mathcal U_i)+\mult(\mathcal V_i)
 &{}<\bigl(a(1+\varepsilon_A)+\kappa
          +N_{\mathcal W}\alpha\bigr)|S_i|+1\\
 &{}=d_0|S_i|+1.
\end{aligned}
\]
Assumption \textup{(iii)} gives $(d-d_0)|S_i|>1$, so, using
\textup{(ii)},
\[
 \mult(\mathcal U_i)+\mult(\mathcal V_i)
 <d_0|S_i|+1<d|S_i|<\frac{m|S_i|}{2}.
\]
This is the hypothesis of Lemma~\ref{lem:naryshkin-payload} with
$N=m|S_i|$.

Define
\[
 Q_i:Y_i\longrightarrow([0,1]^m)^{S_i},
 \qquad
 Q_i(u):=(q_0(su))_{s\in S_i}.
\]
Apply Lemma \ref{lem:naryshkin-payload} with
$N=m|S_i|$, $\mathcal U=\mathcal U_i$, and
$\mathcal V=\mathcal V_i$.    We obtain a
continuous map
$
 \widetilde Q_i
 =
 (\widetilde Q_i^s)_{s\in S_i}:
 Y_i\to([0,1]^m)^{S_i}
$
such that
\begin{equation}\label{eq:payload-block-approximation}
 \|\widetilde Q_i-Q_i\|_\infty<\delta
\end{equation}
and $\widetilde Q_i$ encodes every member of
$\mathcal U_i\cup\mathcal V_i$. 

In particular, encoding $\mathcal U_i$ and its refinement of the
$S_i$-join give
\[
 \widetilde Q_i(u)=\widetilde Q_i(v)
 \quad\Longrightarrow\quad
 \rho(su,sv)<\eta\quad(s\in S_i).
\]
Thus each block map is a local $\eta$-embedding for the orbit
metric $d_{S_i}(u,v):=\max_{s\in S_i}\rho(su,sv)$.

Choose $\chi_i\in C(X,[0,1])$ with
\begin{equation}\label{eq:payload-cutoff}
 \chi_i=1\text{ on }\overline{V_i^+}
 \qquad\text{and}\qquad
 \supp\chi_i\subset B_i.
\end{equation}
Since the sets $sB_i$, $s\in S_i$, are pairwise disjoint, so are
the sets $s\supp\chi_i$. Define the global perturbation by the
cutoff-weighted convex combination
\begin{equation}\label{eq:global-perturbation-formula}
 q(z)
 :=
 \left(
  1-\sum_{i}\sum_{s\in S_i}\chi_i(s^{-1}z)
 \right)q_0(z)+
 \sum_i\sum_{s\in S_i}
 \chi_i(s^{-1}z)\widetilde Q_i^s(s^{-1}z),
\end{equation}
where each weighted term is extended by zero outside $sY_i$.
The sets $sB_i$, over all tower labels $(i,s)$, are pairwise
disjoint.  Hence at every point at most one cutoff is nonzero, and
\eqref{eq:global-perturbation-formula} is a convex combination of two points of
$[0,1]^m$.  It is therefore cube-valued.  Moreover,
$\supp\chi_i\subset B_i$, so every weighted tower term vanishes on a
neighbourhood of the boundary of $sB_i$; its extension by zero is
continuous.  Hence
\begin{equation}\label{eq:payload-uniform-closeness}
 \|q-q_0\|_\infty<\delta<\varepsilon.
\end{equation}
Moreover,
\begin{equation}\label{eq:payload-block-identity}
 q^{S_i}(u)=\widetilde Q_i(u)
 \qquad(u\in V_i^+).
\end{equation}

Since $2\delta<\sigma$, the recognition property is preserved:
\[
 q^F(K_\ell)\cap q^F(X\setminus O_\ell)=\emptyset
 \qquad(\ell\in\Lambda).
\]
Indeed, otherwise equality $q^F(x)=q^F(y)$, with
$x\in K_\ell$ and $y\notin O_\ell$, would imply
\[
 \|q_0^F(x)-q_0^F(y)\|_\infty
 \le2\|q-q_0\|_\infty
 <2\delta<\sigma,
\]
 a contradiction. The preserved recognition property and
  Proposition~\ref{lem:urpc-buffer-router}\textup{(ii)} imply that every pair
$x,y\in X$ satisfying $q^G(x)=q^G(y)$ lies either in a common
tower level $s_0V_i^+$ or in a common remainder chart $U_j^+$.

In the tower case, suppose $x,y\in s_0V_i^+$ and  denote
$u=s_0^{-1}x$ and $v=s_0^{-1}y$. Then $u,v\in V_i^+$.
Equality $q^G(x)=q^G(y)$ gives $q(su)=q(sv)$ for every $s\in S_i$
by using the group element $ss_0^{-1}$. Hence, by
\eqref{eq:payload-block-identity},
\[
 \widetilde Q_i(u)=q^{S_i}(u)=q^{S_i}(v)=\widetilde Q_i(v).
\]
Choose $C\in\mathcal U_i$ with $u\in C$. Since
$\widetilde Q_i^{-1}(\widetilde Q_i(C))=C$, the equality implies
$v\in C$. Refinement of the restricted $S_i$-join gives sets
$W_s\in\mathcal W$, $s\in S_i$, such that
\[
 C\subset Y_i\cap\bigcap_{s\in S_i}s^{-1}W_s.
\]
Taking $s=s_0$, we obtain $x=s_0u,y=s_0v\in W_{s_0}$, and therefore
$\rho(x,y)<\eta$ because $\mesh(\mathcal W)<\eta$.

In the second case that $x,y\in U_j^+$, where $j$ is assigned
to base $i$, we denote
$u=g_jx$ and $v=g_jy$.
Then $u,v\in V_i^+$ and again
$\widetilde Q_i(u)=\widetilde Q_i(v)$.
If $\rho(x,y)\ge\eta$, choose $W\in\mathcal W$ with $x\in W$.
Since $\mesh(\mathcal W)<\eta$, we have $y\notin W$.  Thus
$
 Y_i\cap g_j(U_j^+\cap W)\in\mathcal V_i
$
contains $u$ but not $v$, contradicting the encoding property.
Hence $\rho(x,y)<\eta$ also in this case.

Thus
\[
 q^G(x)=q^G(y)\quad\Longrightarrow\quad \rho(x,y)<\eta,
\]
and so $q^G$ is an $\eta$-embedding.
\end{proof}

\subsection{Proof of the main proposition}
\label{subsec:parameter-assembly}
 We now choose the data in the order required by the preceding
results. In particular, the finite set $F$ is chosen before the marker and
before the pairs supplied by  Proposition~\ref{lem:urpc-buffer-router}.
The observable $q_0$ is chosen only after those pairs are fixed.

\begin{proof}[Proof of Proposition~\ref{prop:eta-approximation}]
Fix
$
 f_{\mathrm{orig}}\in C(X,[0,1]^m),
$ $
 \varepsilon>0
$ and $
 \eta>0.
$
 We shall construct $q$ such that $q^G$ is an $\eta$-embedding and
\[
 \|q-f_{\mathrm{orig}}\|_\infty<\varepsilon.
\]

Choose
$
 \mdim(X,G)<a<\frac m2,
$
and let
$
 c:=(1/2,\ldots,1/2).
$
For sufficiently small $\lambda>0$, the map
$
 f_{\mathrm{in}}
 :=
 (1-\lambda)f_{\mathrm{orig}}+\lambda c
$
takes values in $(0,1)^m$ and satisfies
\begin{equation}\label{eq:main-inward-approximation}
 \|f_{\mathrm{in}}-f_{\mathrm{orig}}\|_\infty
 <\frac{\varepsilon}{4}.
\end{equation}

Choose a finite open cover $\mathcal W$ of $X$ with
$
 \mesh(\mathcal W)<\eta.
$
Apply Lemma~\ref{lem:global-dynamic-pou} to
$\mathcal W$, $f_{\mathrm{in}}$, and $\varepsilon/4$. There exist an honest partition structure $\mathcal H=\hps$, the vectors
$(v_j)_{j=1}^J$, and
$
 f_*=\sum_{j=1}^Jv_j\psi_j.
$
Then
\begin{equation}\label{eq:main-pou-approximation}
 \|f_*-f_{\mathrm{in}}\|_\infty
 <\frac{\varepsilon}{4}.
\end{equation}
Corollary~\ref{cor:pou-window-control} provides
$K_{\mathcal H}\ni e$ and $\delta_{\mathcal H}\in(0,1)$.

By amenability, choose a nonempty finite
$(K_{\mathcal H},\delta_{\mathcal H})$-invariant set $A$ with $e\in A$.   Hence
Corollary~\ref{cor:pou-window-control}, applied to $A$ at the point
$gx$, gives
\begin{equation}\label{eq:main-A-order}
 \sum_{c\in A}\rOrdD{}(cgx)<a|A|
 \qquad(g\in G,\ x\in X).
\end{equation}
Thus the hypothesis of
Proposition~\ref{prop:finite-family-localizer} is satisfied.
Let $F\ni e$ be the finite window supplied by that proposition.

Choose $\varepsilon_A,\kappa>0$ so small that
\begin{equation}\label{eq:main-initial-budget}
 c_0
 :=
 a(1+\varepsilon_A)+\kappa
 <
 \frac m2.
\end{equation}
Fix $\varepsilon_R>0$.  Since $G$ is infinite, choose a nonempty
finite set $R\subset G$ such that
$
 \frac{1+\varepsilon_R}{|R|}<\kappa.
$
By \cite[Lemma~3.18]{naryshkin2024}, the action has the marker
property, as it has URPC.  Denote
\begin{equation}\label{eq:main-marker-test-set}
 \mathcal L
 :=
 \{e\}\cup R^{-1}R\cup F^{-1}F\cup\mathcal Q_F.
\end{equation}
Choose an $\mathcal L$-free  sweep-out set $U_0$.
Shrinking $U_0$ if necessary, choose an open sweep-out set
$U\Subset U_0$ which still has finitely many translates covering
$X$.  Then the choice of $\mathcal L$ ensures that
\begin{equation}\label{eq:main-QF-separation}
\begin{gathered}
 (tU)_{t\in F}\text{ is pairwise disjoint},\\
 (rU)_{r\in R}\text{ is pairwise disjoint},\\
 U\cap qU=\emptyset,\qquad q\in\mathcal Q_F.
\end{gathered}
\end{equation}

Apply Lemma \ref{lem:marker-comparison} to $U$, and let  $\theta>0$ be the resulting constant and $K_U\subset G$ the
resulting finite set.
Apply Lemma~\ref{lem:finite-invariance-propagation} with
$ E=RF^{-1},
$ $
 K=K_{\mathcal H},
 $ $
 \delta=\delta_{\mathcal H},
$
and  $\varepsilon_A,\varepsilon_R$.  This gives
$K_{\mathrm{sp}}\ni e$ and
$\delta_{\mathrm{sp}}\in(0,1)$ such that every sufficiently
$(K_{\mathrm{sp}},\delta_{\mathrm{sp}})$-invariant finite set $S$
satisfies
\begin{equation}\label{eq:main-shape-tests}
 \begin{split}
 |AS|&<(1+\varepsilon_A)|S|,\\
 AS&\text{ is }(K_{\mathcal H},\delta_{\mathcal H})\text{-invariant},\\
 |RF^{-1}S|&<(1+\varepsilon_R)|S|.
 \end{split}
\end{equation}

Set
$
 N_{\mathcal W}:=\mult(\mathcal W).
$
Choose $\alpha>0$ so small that
\begin{equation}\label{eq:main-alpha-choice}
 \alpha<\min\{\theta,\delta_{\mathrm{sp}},1\},
 \qquad
 d_0:=c_0+N_{\mathcal W}\alpha<\frac m2.
\end{equation}
 Choose $d$ such that
\begin{equation}\label{eq:main-d-choice}
 d_0<d<\frac m2.
\end{equation}

Finally choose $N\in\mathbb N$ so large that, whenever $n\ge N$,
\begin{equation}\label{eq:main-shape-size-tests}
 1+\lfloor\alpha n\rfloor
 \le
 \lceil\theta n\rceil,
 \qquad
 n>\frac1{d-d_0}.
\end{equation}

Choose a finite set $E\subset G$ with $|E|\ge N$, and  denote
\[
 K_*
 :=
 K_U\cup K_{\mathrm{sp}}\cup E\cup\{e\}.
\]
Choose a $(K_*,\alpha)$-URPC witness
\[
 \mathfrak W^0
 =
 \bigl(
 \{S_i\},\{V_i^0\},\{g_j\},\{U_j^0\}
 \bigr).
\]

Since $S_i$ is $(K_*,\alpha)$-invariant and
$\alpha<1$,
$
 \left|
 \bigcap_{k\in K_*}kS_i
 \right|
 >
 (1-\alpha)|S_i|
 >0.
$
Hence we may choose
$
 h_i\in\bigcap_{k\in K_*}kS_i.
$
Since $E\subset K_*$, for every $k\in E$ we have
$
 h_i\in kS_i,
$
and therefore
$
 k^{-1}h_i\in S_i.
$
Thus the map
$
 E\to S_i,
$ $
 k\mapsto k^{-1}h_i,
$
is well defined and injective.  Hence
\[
 |S_i|\ge |E|\ge N.
\]
Moreover,
$K_U,K_{\mathrm{sp}}\subset K_*$ and
$\alpha<\min\{\theta,\delta_{\mathrm{sp}}\}$, so every $S_i$
satisfies the three conditions in
\eqref{eq:main-shape-tests}, as well as
\eqref{eq:main-shape-size-tests}.
Thus the same witness simultaneously satisfies all shape
requirements of   Proposition~\ref{lem:urpc-buffer-router},
Lemma~\ref{lem:scale-uniform-coordinates}, and
Proposition~\ref{prop:global-perturbation}.

Apply  Proposition~\ref{lem:urpc-buffer-router} to this witness.
We obtain nested castles and remainder sets together with a finite family of 
routing pairs
\[
 K_\ell\subset O_\ell\Subset U,
 \qquad
 \ell\in\Lambda,
\]
 with the closures $\overline{O_\ell}$ pairwise disjoint.

The disjointness conditions on translates of $U$ in
\eqref{eq:main-QF-separation} allow us to apply
Proposition~\ref{prop:finite-family-localizer} to this finite family
with approximation tolerance $\varepsilon/4$.
Choose the resulting localizer $q_0\in C(X,[0,1]^m)$ and the
functions $(\rho_\ell)_{\ell\in\Lambda}$ as in
Remark~\ref{rem:prepared-localizer}. Then
\begin{equation}\label{eq:main-localizer-approximation}
 \|q_0-f_*\|_\infty<\frac{\varepsilon}{4},
\end{equation}
and $q_0$ satisfies the finite-window recognition condition
\begin{equation}\label{eq:main-localizer-separation}
 q_0^F(K_\ell)\cap q_0^F(X\setminus O_\ell)=\emptyset
 \qquad(\ell\in\Lambda).
\end{equation}

All hypotheses of Proposition~\ref{prop:global-perturbation} are now
satisfied.   Applying that proposition with
$\varepsilon/4$, we obtain
$q\in C(X,[0,1]^m)$ such that
\begin{equation}\label{eq:main-payload-approximation}
 \|q-q_0\|_\infty<\frac{\varepsilon}{4}
\end{equation}
and $q^G$ is an $\eta$-embedding.

Combining
\eqref{eq:main-inward-approximation},
\eqref{eq:main-pou-approximation},
\eqref{eq:main-localizer-approximation}, and
\eqref{eq:main-payload-approximation}, we obtain
\[
 \|q-f_{\mathrm{orig}}\|_\infty<\varepsilon.\qedhere
\]
\end{proof}

\section{Proof of Proposition~\ref{prop:finite-family-localizer}}\label{section:main}
\label{sec:recognition-proof}
The preceding section proves
Proposition~\ref{prop:eta-approximation}, and hence
Theorem~\ref{thm:main}, assuming the finite-window localization
result in Proposition~\ref{prop:finite-family-localizer}.
We now prove that proposition.
We construct perturbations of the form
$$
 q_\Theta(z)_j
 =
 f_*(z)_j+\langle\Xi(z),\theta^{(j)}\rangle,
 \qquad 1\le j\le m,
$$
where
$$
 \Theta=(\theta^{(1)},\ldots,\theta^{(m)})
 \in(\mathbb R^P)^m.
$$
The coordinate map $\Xi:X\to\mathbb R^P$ is built from
translated cutoff functions and polynomial coordinates of the
partition-of-unity map. Once $\Xi$ is fixed, choosing $\Theta$
sufficiently close to zero makes $q_\Theta$ uniformly close
to $f_*$.

This parametrization turns a failure of the finite-window
recognition condition into a system of linear equations in
the perturbation parameters $\theta^{(j)}$.
We describe the configurations giving rise to such failures
using finitely many semialgebraic parameter spaces.
By comparing their dimensions with the ranks of the
corresponding linear systems, we show that the set of
perturbation parameters for which recognition fails has
nowhere dense closure. We can therefore choose parameters
arbitrarily close to zero for which recognition holds.
We begin with the algebraic estimates needed for this argument.

\subsection{Semialgebraic avoidance and polynomial evaluations}

The next lemma is essentially a linear algebra rank counting argument.

\begin{lemma}
\label{lem:semialgebraic-incidence-avoidance}
Let $Z$ be a semialgebraic set with
$
 \dim Z\le D,
$
and let
\[
 M:Z\longrightarrow\operatorname{Mat}_{n\times P}(\mathbb R),
 \qquad
 c_\ell:Z\longrightarrow\mathbb R^n,
 \quad 1\le\ell\le m,
\]
be semialgebraic maps.  Suppose that
$
 \rank M(z)=r
$ for all $z\in Z.$
If
$
 D<mr,
$
then the bad parameter set
\[
 \mathcal B
 :=
 \left\{
 (\theta_1,\ldots,\theta_m)\in(\mathbb R^P)^m:
 \begin{array}{l}
 \text{there exists }z\in Z\text{ such that}\\
 M(z)\theta_\ell=c_\ell(z)
 \text{ for every }1\le\ell\le m
 \end{array}
 \right\}
\]
is semialgebraic and has nowhere dense closure in
$(\mathbb R^P)^m$.

More generally, let
$
 Z=\bigsqcup_{\alpha=1}^N Z_\alpha
$
be a finite semialgebraic partition such that $M$ has constant rank
$r_\alpha$ on $Z_\alpha$.  If
$\dim Z_\alpha<mr_\alpha$ for $1\le\alpha\le N$,
then the same conclusion holds.
\end{lemma}

\begin{proof}
Consider the  set
\[
 \mathcal I
 :=
 \left\{
 (z,\theta_1,\ldots,\theta_m)
 \in Z\times(\mathbb R^P)^m:
 M(z)\theta_\ell=c_\ell(z)
 \text{ for every }1\le\ell\le m
 \right\},
\]
which is semialgebraic.

Let
$
 \pi_Z:\mathcal I\to Z
$
be the coordinate projection.  Fix $z\in Z$.  If
$\pi_Z^{-1}(z)$ is nonempty, then, for each $1\le\ell\le m$, the
solution set of
$M(z)\theta_\ell=c_\ell(z)$
is an affine subspace of $\mathbb R^P$ of dimension $P-r$.
Consequently,
$$\dim \pi_Z^{-1}(z)=m(P-r).$$
By Lemma~\ref{lem:basic-semialgebraic-facts} (v) and the assumption
$D<mr$,
 \begin{align}
 \begin{aligned}
 \dim\mathcal I
 &\le \dim\pi_Z(\mathcal I)+m(P-r)\le \dim Z+m(P-r)\\
 &\le D+m(P-r)<mP.
 \end{aligned}
 \end{align}

Let
\[
 \pi_\theta:
 Z\times(\mathbb R^P)^m
 \longrightarrow
 (\mathbb R^P)^m
\]
be the parameter projection.  Then
$
 \mathcal B=\pi_\theta(\mathcal I).
$
By Lemma~\ref{lem:basic-semialgebraic-facts} (iv),
$\mathcal B$ is semialgebraic and
$\dim\mathcal B \le \dim\mathcal I <mP$.
Lemma~\ref{lem:basic-semialgebraic-facts} (vi) implies that
$\mathcal B$ has nowhere dense closure.

For the final assertion, apply the preceding argument to each
$Z_\alpha$.  If $\mathcal B_\alpha$ denotes the corresponding bad
parameter set, then
$\dim\mathcal B_\alpha<mP$ for $1\le\alpha\le N$.
Since
$
 \mathcal B=\bigcup_{\alpha=1}^N\mathcal B_\alpha,
$
the finite-union property in
Lemma~\ref{lem:basic-semialgebraic-facts} (i) gives
\[
 \dim\mathcal B
 =
 \max_{1\le\alpha\le N}\dim\mathcal B_\alpha
 <mP.
\]
Hence $\mathcal B$ again has nowhere dense closure.
\end{proof}

 {We next show that evaluations of sufficiently many monomials at
distinct points give linearly independent vectors.}

\begin{lemma}
\label{lem:finite-veronese}
Let $z_1,\ldots,z_N$ be distinct points of $\Delta^{J-1}$.  If
$d\ge N-1$, then the degree-$d$ Veronese evaluation vectors
$\Phi_d(z_1),\ldots,\Phi_d(z_N)$ are linearly independent.
\end{lemma}
\begin{proof}
Suppose that
$$
 \sum_{i=1}^N c_i\Phi_d(z_i)=0.
$$
Fix $j\in\{1,\ldots,N\}$. Since $N\le d+1$, the interpolation
property \eqref{eq:interpolation} provides coefficients
$a_{j,\alpha}\in\mathbb R$ such that the polynomial
$$
 P_j(z):=\sum_{\alpha\in\mathcal A_{J,d}}a_{j,\alpha}z^\alpha
$$
satisfies $P_j(z_i)=\delta_{ij}$ for every $i$.
Taking the scalar product of the assumed vector identity with
$a_j:=(a_{j,\alpha})_{\alpha\in\mathcal A_{J,d}}$, we obtain
$$
 0=\sum_{i=1}^N c_i\langle a_j,\Phi_d(z_i)\rangle
  =\sum_{i=1}^N c_iP_j(z_i)=c_j.
$$
Since $j$ was arbitrary, all coefficients vanish.
\end{proof}
\subsection{Shifted differences and separated marker translates}
 {The following two lemmas give a rank bound for shifted coordinate
differences and a uniqueness statement for visits to translates of a
separated set. They will be used when comparing two finite orbit names.}

\begin{lemma}
\label{lem:anchor-rank}
 {Let $G$ be a group, let $F\subset G$ be nonempty and finite, let
$h\in G\setminus\{e\}$, and let $\beta\in[0,1]$.
In $\mathbb R^F$, let $(e_t)_{t\in F}$ denote the standard basis and
consider}
\[
 a_s
 :=
 e_{s}
 -
 \begin{cases}
 \beta e_{sh},& sh\in F,\\
 0,& sh\notin F,
 \end{cases}
 \qquad s\in F.
\]
Then
\[
 \rank\{a_s:s\in F\}\ge \frac{|F|}{2}.
\]
Moreover, if $\beta<1$, then the vectors $a_s$, $s\in F$, are
linearly independent.
\end{lemma}

\begin{proof}
Consider the directed graph with vertex set $F$ and with an edge
\[
 s\longrightarrow sh
\]
whenever $s,sh\in F$.  Since right multiplication by $h$ is
injective, every vertex has at most one incoming edge and at most one
outgoing edge.  Hence every connected component is either a directed
path or a directed cycle.

First consider a path
\[
 s_1\longrightarrow s_2\longrightarrow\cdots
 \longrightarrow s_r,
\]
so that
$s_i h=s_{i+1}$ for $1\le i<r$, and $s_rh\notin F$.
The corresponding vectors are
\[
 \begin{aligned}
 a_{s_1}&=e_{s_1}-\beta e_{s_2},\\
 a_{s_2}&=e_{s_2}-\beta e_{s_3},\\
 &\hspace{1.2cm}\vdots\\
 a_{s_{r-1}}&=e_{s_{r-1}}
              -\beta e_{s_r},\\
 a_{s_r}&=e_{s_r}.
 \end{aligned}
\]
Relative to the ordered basis
$
e_{s_1},\ldots,e_{s_r},
$
their coefficient matrix is upper triangular with diagonal entries
equal to $1$.  Thus a path with $r$ vertices contributes rank
$r$.

Now consider a directed cycle
\[
 s_1\longrightarrow s_2\longrightarrow\cdots
 \longrightarrow s_r\longrightarrow s_1.
\]
Relative to the same ordering of the vertices, the corresponding
matrix is
$
 I_r-\beta P_r,
$
where
\[
 P_r=
 \begin{pmatrix}
 0&1&0&\cdots&0\\
 0&0&1&\cdots&0\\
 \vdots&&&\ddots&\vdots\\
 0&0&0&\cdots&1\\
 1&0&0&\cdots&0
 \end{pmatrix}
\]
is the cyclic permutation matrix.  Since
$\det(I_r-\beta P_r)=1-\beta^r$,
the cycle contributes full rank $r$ whenever $\beta<1$.

If $\beta=1$, then
$
 I_r-P_r
$
has one-dimensional kernel, spanned by
$
 (1,\ldots,1),
$
and therefore has rank $r-1$.  Thus, when $\beta=1$, each cycle
causes a loss of exactly one rank.

Finally, since $h\ne e$, there are no cycles of length one:
$sh=s$ would imply $h=e$.  Hence every cycle has at least two
vertices.  If $c$ denotes the number of cycle components, then
$
 c\le \frac{|F|}{2}.
$
Summing the rank contributions of all path and cycle components gives
$\rank\{a_s:s\in F\} \ge |F|-c \ge \frac{|F|}{2}$.
When $\beta<1$, no component loses rank, so the vectors
$a_s$, $s\in F$, are linearly independent.
\end{proof}

\begin{lemma}
\label{lem:anchor-synchronization}
Let $F\subset G$ be finite and    {denote}
\[
 \mathcal Q_F
 :=
 \{t_2^{-1}s_2s_1^{-1}t_1:
   s_1,s_2,t_1,t_2\in F\}
 \setminus\{e\}.
\]
Let $U\subset X$ satisfy
$
 U\cap qU=\emptyset
$, $q\in\mathcal Q_F.$ Let $y\in X$.
Suppose that for $i=1,2$, there exist $v_i\in U$ and $s_i,t_i\in F$ such that $ s_i y=t_i v_i$.
Then
$v_1=v_2$ and $s_1^{-1}t_1=s_2^{-1}t_2$.
In particular, if $v_i\in O_{k_i}$, where the sets
$(O_k)_{k\in\Lambda}$ are pairwise disjoint subsets of $U$, then
$
 k_1=k_2.
$
\end{lemma}

\begin{proof}
From
$
 s_1y=t_1v_1,
$ and $
 s_2y=t_2v_2,
$
we obtain
$v_2=qv_1$,
where
$
 q:=t_2^{-1}s_2s_1^{-1}t_1.
$
If $q\ne e$, then $q\in\mathcal Q_F$, while
$v_2=qv_1\in U\cap qU$,
a contradiction.  Hence $q=e$, which gives
$v_1=v_2$ and $s_1^{-1}t_1=s_2^{-1}t_2$.
The final assertion follows from the disjointness of the $O_k$'s.
\end{proof}

\subsection{Coordinate values and their dimension bound}
 {We now record the distinct partition-coordinate values occurring
in two orbit names. The notation in this subsection is local to the
following lemma.}

Let $D_1,\ldots,D_J$ be a finite closed cover of $X$, and let
\[
 \boldsymbol{\psi}=(\psi_1,\ldots,\psi_J):X\longrightarrow\Delta^{J-1}
\]
be a partition of unity satisfying
$
 \supp\psi_j\subset D_j
 $, $1\le j\le J.
$
  {Write $\mathcal D:=(D_j)_{j=1}^J$.}
For $z\in X$, recall
$\rLocD{}(z)=\{j:z\in D_j\}$ and $\rOrdD{}(z)=|\rLocD{}(z)|-1$.
Then
\begin{equation}\label{eq:obser.}
     \boldsymbol{\psi}(z)\in\Delta_{\rLocD{}(z)}.
\end{equation}

Fix $x,y\in X$ and finite
sets $A,F\subset G$, with $e\in A$.    {Denote}
\[
 H:=AF,
 \qquad
 H^\circ:=\bigcap_{c\in A}cF.
\]
For $u\in H$, write
$p_u:=\boldsymbol{\psi}(ux)$ and $q_u:=\boldsymbol{\psi}(uy)$.

We call $u\in H$ a \textbf{non-loop site}  if
$
 p_u\ne q_u.
$
Let $\mathscr V$ be the set of all values among the $p_u$'s and
$q_u$'s which occur at a non-loop site, and let
\[
 \mathscr V^\circ
 :=
 \{v\in\mathscr V:
   v=p_u\text{ or }v=q_u
   \text{ for some non-loop }u\in H^\circ\}.
\]

For $v\in\mathscr V$, define
\[
 I(v)
 :=
 \bigcap_{\{u:p_u=v\}} \rLocD{}(ux)
 \cap
 \bigcap_{\{u:q_u=v\}} \rLocD{}(uy),
\]
where an empty intersection is omitted, and    {denote}
$d(v):=|I(v)|-1$.
Notice that $I(v)\ne\emptyset$, since by \eqref{eq:obser.},
\[
 v\in
 \bigcap_{\{u:p_u=v\}}\Delta_{\rLocD{}(ux)}
 \cap
 \bigcap_{\{u:q_u=v\}}\Delta_{\rLocD{}(uy)}
 =
 \Delta_{I(v)}.
\]

The matrix below records the differences that appear in the
collision equations. For each $c\in A$, it assigns a coefficient
$+1$ to $p_{cs}$ and $-1$ to $q_{cs}$ whenever these values belong
to $\mathscr V^\circ$; if the two values coincide, their
contributions cancel. Its rank will provide a lower bound for
the rank of the coefficient matrix in the linear equations
expressing failure of finite-window recognition.

Finally, define the \textbf{core incidence matrix} $B^\circ$, with rows
indexed by $s\in F$ and columns indexed by
$(c,v)\in A\times\mathscr V^\circ$, by
\[
 B^\circ_{s,(c,v)}
 :=
 \1_{\{p_{cs}=v\}}
 -
 \1_{\{q_{cs}=v\}},
\]
and set
$r_\circ:=\rank B^\circ$.

 {A \textbf{state} is an element $v\in\mathscr V$, that is, a distinct
value occurring at a non-loop site. It is a \textbf{core state} if
$v\in\mathscr V^\circ$. The number
$d(v)=\dim\Delta_{I(v)}$ is the dimension of its common coordinate
face. The next lemma bounds the sum of these dimensions: the core
contribution is bounded by the rank of $B^\circ$, and the remaining
contribution by $|H\setminus H^\circ|$.}

\begin{lemma}
\label{lem:state-rank-dimension}
   {Denote} $b_F:=|H\setminus H^\circ|$.  Assume that, for every $s\in F$,
\begin{equation}\label{eq:state-order-assumption}
 \sum_{c\in A} \rOrdD{}(csx)<a|A|,
 \qquad
 \sum_{c\in A} \rOrdD{}(csy)<a|A|.
\end{equation}
Then
\[
 \sum_{v\in\mathscr V^\circ}d(v)
 \le 2ar_\circ
\qquad\text{and}\qquad
 \sum_{v\in\mathscr V\setminus\mathscr V^\circ}d(v)
 \le 2(J-1)b_F.
\]
Consequently,
\[
 \sum_{v\in\mathscr V}d(v)
 \le
 2ar_\circ+2(J-1)b_F.
\]
\end{lemma}

\begin{proof}
We first prove that
\begin{equation}\label{eq:eq_1}
    \sum_{v\in\mathscr V^\circ}d(v)
 \le 2ar_\circ.
\end{equation}

Fix $v\in\mathscr V^\circ$ and $c\in A$.  Choose a non-loop site
$u\in H^\circ$ such that
$
 v=p_u
$ or $
 v=q_u.
$
Since $u\in H^\circ\subset cF$,
$
 s:=c^{-1}u\in F.
$
As $p_u\ne q_u$, exactly one of $p_u,q_u$ equals $v$.  Hence
$B^\circ_{s,(c,v)}=\pm1$.
Thus every column of $B^\circ$ is nonzero.

If $r_\circ=0$, then $B^\circ$ is the zero matrix. This is possible only when
$\mathscr V^\circ=\emptyset$. Thus, \eqref{eq:eq_1} holds.

Suppose now that $r_\circ>0$.  Choose
$R_0\subset F$, with $|R_0|=r_\circ$, so that the rows indexed by
$R_0$ form a basis of the row space of $B^\circ$.
For every column $(c,v)$, there exists $s\in R_0$ such that
$
 B^\circ_{s,(c,v)}\ne0.
$
Indeed, otherwise that column would vanish on all basis rows, and
hence on every row of $B^\circ$, contradicting the fact that it is
nonzero. 

For each column $(c,v)$, fix
$
 s(c,v)\in R_0
$
such that
$
 B^\circ_{s(c,v),(c,v)}\ne0.
$

For $s\in R_0$ and $c\in A$,    {denote}
\[
 \mathscr V_{s,c}
 :=
 \{v\in\mathscr V^\circ:s(c,v)=s\}.
\]
  {For each fixed $c\in A$, every $v\in\mathscr V^\circ$ is assigned
exactly one index $s(c,v)\in R_0$. Thus these sets are pairwise
disjoint and cover $\mathscr V^\circ$, so}
\begin{equation}\label{eq:disjoin}
    \mathscr V^\circ=\bigsqcup_{s\in R_0}\mathscr V_{s,c}.
\end{equation}
By the definition of $B^\circ$,
$
 B^\circ_{s,(c,v)}
 =
 \1_{\{p_{cs}=v\}}
 -
 \1_{\{q_{cs}=v\}},
$
  {and for $v\in\mathscr V_{s,c}$ this entry is nonzero by the
choice of $s(c,v)$. Therefore}
$
 \mathscr V_{s,c}\subset\{p_{cs},q_{cs}\}.
$
Notice also that if $p_{cs}=q_{cs}$, then
$\mathscr V_{s,c}=\emptyset$.

If $v=p_{cs}$,  {the occurrence $(cs,+)$ is among those used to define
$I(v)$, and hence}
$
 I(v)\subset \rLocD{}(csx),
$
and hence
$
 d(v)\le \rOrdD{}(csx).
$
Similarly, if $v=q_{cs}$, then
 {$I(v)\subset \rLocD{}(csy)$ and hence}
$
 d(v)\le \rOrdD{}(csy).
$
So
\[
 \sum_{v\in\mathscr V_{s,c}}d(v)
 \le
 \rOrdD{}(csx)+\rOrdD{}(csy).
\]

By \eqref{eq:disjoin},
 {summing over $c\in A$ and applying
the assumed order bounds for each $s\in R_0$,}
one has
 {
\begin{align*}
 |A|\sum_{v\in\mathscr V^\circ}d(v)
 &=\sum_{c\in A}\sum_{s\in R_0}
   \sum_{v\in\mathscr V_{s,c}}d(v)\\
 &\le\sum_{s\in R_0}\sum_{c\in A}
   \bigl(\rOrdD{}(csx)+\rOrdD{}(csy)\bigr)\\
 &<\sum_{s\in R_0}2a|A|
 =2a|A|r_\circ.
\end{align*}
For the strict inequality, apply
\eqref{eq:state-order-assumption} to each $s\in R_0$; the final
identity uses $|R_0|=r_\circ$.
}
Dividing by $|A|$ gives
\[
 \sum_{v\in\mathscr V^\circ}d(v)
 \le 2ar_\circ.
\]

It remains to estimate the values occurring only outside the core.
For each
$
 v\in\mathscr V\setminus\mathscr V^\circ,
$
choose a non-loop site
$
 u_v\in H\setminus H^\circ
$
and one of its two endpoints at which $v$ occurs.  This gives an
injective map
\[
 \mathscr V\setminus\mathscr V^\circ
 \longrightarrow
 (H\setminus H^\circ)\times\{+,-\}.
\]
Hence
\[
 |\mathscr V\setminus\mathscr V^\circ|
 \le
 2|H\setminus H^\circ|.
\]
Since $d(v)\le J-1$ for every $v\in\mathscr V$, we obtain
\[
 \sum_{v\in\mathscr V\setminus\mathscr V^\circ}d(v)
 \le
 2(J-1)|H\setminus H^\circ|.\qedhere
\]
\end{proof}

\subsection{Construction of the perturbation and exclusion of collisions}
\label{subsec:recognition-construction}
We return to the hypotheses of
Proposition~\ref{prop:finite-family-localizer}. The proof first chooses $F$
and defines the perturbation family. It then describes failed recognition
by finitely many types of equations and applies
Lemma~\ref{lem:semialgebraic-incidence-avoidance}.

\begin{proof}[Proof of Proposition~\ref{prop:finite-family-localizer}]
Since $a<m/2$, choose $\eta_F>0$ such that
\begin{equation}\label{eq:localizer-eta-choice}
 2(J-1)\eta_F<\frac{m-2a}{4}.
\end{equation}
Amenability of $G$ gives a nonempty finite set
$F$ with $e\in F$ such that
\begin{equation}\label{eq:localizer-F-choice}
 \sum_{c\in A}|cF\triangle F|<\eta_F|F|
 \qquad\text{and}\qquad
 |F|>\frac{4}{m-2a}.
\end{equation}
Set
\[
 H:=AF,
 \qquad
 H^\circ:=\bigcap_{c\in A}cF,
 \qquad
 b_F:=|H\setminus H^\circ|.
\]
Since $e\in A$, one has
\begin{equation}\label{eq:localizer-boundary-inclusion}
 H\setminus H^\circ
 \subset
 \bigcup_{c\in A}(cF\triangle F).
\end{equation}
Indeed, let $g\in H\setminus H^\circ$. If $g\notin F$, then
$g\in cF\setminus F$ for some $c\in A$. If $g\in F$, then
$g\notin H^\circ$ gives $g\in F\setminus cF$ for some $c\in A$.
In either case $g\in cF\triangle F$. Consequently,
\[
 b_F\le\sum_{c\in A}|cF\triangle F|<\eta_F|F|.
\]
It follows from
\eqref{eq:localizer-F-choice} that $b_F<\eta_F|F|$. Together with
\eqref{eq:localizer-eta-choice} and the lower bound on $|F|$, this gives
\begin{equation}\label{eq:localizer-boundary-budget}
 2(J-1)b_F+1
 <\frac{m-2a}{2}|F|.
\end{equation}
 {More explicitly, the lower bound on $|F|$ gives
$1<(m-2a)|F|/4$, and therefore
\[
 2(J-1)b_F+1\le 2(J-1)\eta_F|F|+1<\frac{m-2a}{4}|F|+1
 <\frac{m-2a}{2}|F|.
\]}

For each $\ell\in\Lambda$, choose a continuous function
$\rho_\ell:X\to[0,1]$ such that
\begin{equation}\label{eq:localizer-bump-functions}
 \rho_\ell=1\ \text{on }K_\ell,
 \qquad
 \supp\rho_\ell\subset O_\ell.
\end{equation}
For $t\in F$ and $\ell\in\Lambda$, define
\begin{equation}\label{eq:localizer-anchor-functions}
 b_{t,\ell}(z)
 :=
  \rho_\ell(t^{-1}z),\qquad z\in X,
\end{equation}
This is continuous.  

Choose an integer $d_V\ge 2|H|-1$, and let
$\Phi_{d_V}:\Delta^{J-1}\to\mathbb R^{M_{J,d_V}}$ be the Veronese map
(see Subsection \ref{subsec:ve}).   {Denote}
$P:=|F||\Lambda|+|A|M_{J,d_V}$
and
 {
\[
 \Xi:X\to\mathbb R^P,\qquad
 \Xi(z):=
 \left(
  (b_{t,\ell}(z))_{(t,\ell)\in F\times\Lambda},
  (\Phi_{d_V}(\boldsymbol{\psi}(cz)))_{c\in A}
 \right),\quad z\in X.
\]
}
For
$
 \Theta=(\theta^{(1)},\ldots,\theta^{(m)})
 \in(\mathbb R^P)^m,
$
define
\begin{equation}\label{eq:localizer-parametric-map}
 q_\Theta(z)_j
 :=f_*(z)_j+\langle\Xi(z),\theta^{(j)}\rangle,
 \qquad 1\le j\le m.
\end{equation}
We shall show that the set of $\Theta$ for which
\eqref{eq:localizer-separation} fails has nowhere dense closure.

Fix $\ell\in\Lambda$, $x\in K_\ell$, and
 {$y\in X\setminus O_\ell$. For $s\in F$ and $1\le j\le m$,
\eqref{eq:localizer-parametric-map} gives
\[
 q_\Theta(sx)_j-q_\Theta(sy)_j
 =f_*(sx)_j-f_*(sy)_j
 +\langle\Xi(sx)-\Xi(sy),\theta^{(j)}\rangle.
\]
Thus a collision gives linear equations in the vectors
$\theta^{(j)}$. We now describe their coefficients.
For $s\in F$, define the anchor row}
\[
 A_s(x,y)
 :=
 \bigl(b_{t,k}(sx)-b_{t,k}(sy)\bigr)_{(t,k)\in F\times\Lambda}.
\]
Since $x\in K_\ell$, \eqref{eq:localizer-bump-functions} gives
$b_{s,\ell}(sx)=1$, since
$b_{s,\ell}(sx)=\rho_\ell(s^{-1}sx)=\rho_\ell(x)$. If $b_{t,k}(sx)\ne0$, then
 {$\rho_k(t^{-1}sx)\ne0$ and hence} $sx\in tO_k\subset tU$ and $sx\in sU$. The pairwise disjointness of
$(tU)_{t\in F}$ gives $t=s$, and then
$x\in O_k\cap O_\ell$. Since the closures of the $O_j$'s are
pairwise disjoint, $k=\ell$. Therefore
\begin{equation}\label{eq:localizer-first-anchor}
 (b_{t,k}(sx))_{(t,k)\in F\times\Lambda}=e_{(s,\ell)}.
\end{equation}
Here $e_{(s,\ell)}\in \mathbb R^{F\times\Lambda}$ denotes the
standard coordinate vector indexed by $(s,\ell)$.

Suppose first that
$
 b_{t,k}(sy)=0
 $ $s,t\in F,\ k\in\Lambda.$
Then $A_s(x,y)=e_{(s,\ell)}$ for every $s\in F$, so these rows are
linearly independent. Suppose instead that at least one of the displayed
coefficients is nonzero.  {Choose $s_0,t_0\in F$ and $k\in\Lambda$ such that
$b_{t_0,k}(s_0y)>0$, and    {denote}
\begin{equation}\label{eq:v}
\begin{gathered}
     h:=s_0^{-1}t_0\in F^{-1}F,\qquad  v:=h^{-1}y\in O_k,
 \qquad
 \beta:=\rho_k(v)\in(0,1].
\end{gathered}
\end{equation}
Here $\rho_k(v)=b_{t_0,k}(s_0y)>0$ and $s_0y=t_0v$.} Thus $y=hv$. If $b_{t',k'}(s'y)\ne0$, then
$s'y=t'v'$ for some $v'\in O_{k'}$. By
Lemma~\ref{lem:anchor-synchronization} and
\eqref{eq:localizer-marker-separation},
\begin{equation}\label{eq:827}
    v'=v, \qquad s'^{-1}t'=h, \qquad \text{and} \qquad k'=k.
\end{equation}
Thus the only possible nonzero anchor coefficient at $sy$ is the
coordinate $(sh,k)$; it occurs exactly when $sh\in F$, and its value
is $\beta$.  {Indeed, \eqref{eq:827} applies to every nonzero coefficient,
so its indices must satisfy $t=sh$ and $k'=k$. When $sh\in F$,
\[
 b_{sh,k}(sy)=\rho_k\bigl((sh)^{-1}sy\bigr)
 =\rho_k(h^{-1}y)=\rho_k(v)=\beta.
\]}
Hence
\begin{equation}\label{eq:localizer-anchor-row}
 A_s(x,y)
 =e_{(s,\ell)}-
 \begin{cases}
  \beta e_{(sh,k)},&sh\in F,\\
  0,&sh\notin F.
 \end{cases}
\end{equation}

We now encode, with exact notation, the finitely many discrete records
that can be produced by pairs
$(x,y)\in K_\ell\times(X\setminus O_\ell)$. Let
$\mathscr E_H:=H\times\{+,-\}$.
For such a pair define, for $u\in H$,
\[
 \xi^{x,y}_{u,+}:=\boldsymbol{\psi}(ux),
 \qquad
 \xi^{x,y}_{u,-}:=\boldsymbol{\psi}(uy),
\]
and
\[
 I^{x,y}_{u,+}:=\rLocD{}(ux),
 \qquad
 I^{x,y}_{u,-}:=\rLocD{}(uy).
\]
Define an equivalence relation $\sim_{x,y}$ on $\mathscr E_H$ by
\[
 (u,\sigma)\sim_{x,y}(u',\sigma')
 \quad\Longleftrightarrow\quad
 \xi^{x,y}_{u,\sigma}=\xi^{x,y}_{u',\sigma'}.
\]
The pair also determines a discrete anchor type $\tau_{\ell,x,y}$, which
is the symbol $\emptyset$ if all the coefficients
$b_{t,k}(sy)$, $s,t\in F$, $k\in\Lambda$, vanish; otherwise \eqref{eq:827} gives a unique pair
$(k,h)\in\Lambda\times F^{-1}F$, and we set
$\tau_{\ell,x,y}:=(k,h)$. In the latter case
$(k,h)\ne(\ell,e)$.
 {Indeed, if $k=\ell$ and $h=e$, then $y=v\in O_\ell$,
contrary to the choice of $y$. Thus, in the subcase $k=\ell$, this
condition is simply $h\ne e$.
We call $\tau_{\ell,x,y}=\emptyset$ the \textbf{no-hit case} and
$\tau_{\ell,x,y}=(k,h)$ the \textbf{hit case}.
Geometrically, $b_{t,k}(sy)>0$ exactly when
$sy\in t\{z\in X:\rho_k(z)>0\}$.} 

Define
\[
 \omega(\ell,x,y)
 :=
 \left(
  \ell,
  (I^{x,y}_{u,\sigma})_{(u,\sigma)\in\mathscr E_H},
  \sim_{x,y},
  \tau_{\ell,x,y}
 \right)
\]
and let
\[
 \Omega
 :=
 \left\{
  \omega(\ell,x,y):
  \ell\in\Lambda,\ x\in K_\ell,\ y\in X\setminus O_\ell
 \right\}.
\]
The set $\Omega$ is finite. Indeed, $\ell$ ranges over the finite set
$\Lambda$; each $I^{x,y}_{u,\sigma}$ is one of the finitely many
nonempty subsets of $\{1,\ldots,J\}$; there are only finitely many
equivalence relations on the finite set $\mathscr E_H$; and the anchor
type belongs to the finite set
$
 \{\emptyset\}
 \cup
 \bigl(\Lambda\times F^{-1}F\bigr).
$

Fix $\omega\in\Omega$. Write its components as
\[
 \omega
 =
 \left(
  \ell_\omega,
  (I^\omega_{u,\sigma})_{(u,\sigma)\in\mathscr E_H},
  \sim_\omega,
  \tau_\omega
 \right),
\]
and let
$C_\omega(u,\sigma):=[(u,\sigma)]_{\sim_\omega}$
denote the corresponding equivalence class. Since $\omega$ is produced
by an actual pair, for all $s\in F$
\begin{equation}\label{eq:localizer-admissible-pattern}
\begin{split}
   &  \sum_{c\in A}(|I^\omega_{cs,+}|-1)=\sum_{c\in A}\rOrdD{}(cs x)<a|A|,
\\
 &\sum_{c\in A}(|I^\omega_{cs,-}|-1)=\sum_{c\in A}\rOrdD{}(cs y)<a|A|.
\end{split}
\end{equation}

Define the set of non-loop sites of $\omega$ by
\[
 \mathscr N_\omega
 :=
 \{u\in H:C_\omega(u,+)\ne C_\omega(u,-)\}.
\]
Let
\[
 \mathscr C_\omega
 :=
 \{C_\omega(u,\sigma):u\in\mathscr N_\omega,\ \sigma\in\{+,-\}\}
\]
be the set of equivalence classes that occur at a non-loop site, and    {denote}
\[
 \mathscr C_\omega^\circ
 :=
 \{C_\omega(u,\sigma):u\in\mathscr N_\omega\cap H^\circ,
                          \ \sigma\in\{+,-\}\}.
\]
For $C\in\mathscr C_\omega$, define
\begin{equation}\label{eq:localizer-class-carrier}
 I_\omega(C)
 :=
 \bigcap_{(u,\sigma)\in C}I^\omega_{u,\sigma},
 \qquad
 d_\omega(C):=|I_\omega(C)|-1.
\end{equation}
The set $I_\omega(C)$ is nonempty. Indeed, choose a pair $(x,y)$
producing $\omega$, and let $\xi_C^{x,y}$ be the common value of
$\xi^{x,y}_{u,\sigma}$ for $(u,\sigma)\in C$.  {Since
$\xi_C^{x,y}\in\Delta^{J-1}$, choose $j$ with $(\xi_C^{x,y})_j>0$.
For every $(u,\sigma)\in C$, membership in
$\Delta_{I^\omega_{u,\sigma}}$ implies $j\in I^\omega_{u,\sigma}$.
Hence $j\in I_\omega(C)$, proving nonemptiness.} Then
\[
 \xi_C^{x,y}
 \in
 \bigcap_{(u,\sigma)\in C}\Delta_{I^\omega_{u,\sigma}}
 =\Delta_{I_\omega(C)}.
\]
Consequently,
\begin{equation}\label{eq:localizer-class-dimension}
 d_\omega(C)=\dim\Delta_{I_\omega(C)}\ge0.
\end{equation}

Define the matrix $B_\omega^\circ$, with rows indexed by $s\in F$
and columns indexed by $(c,C)\in A\times\mathscr C_\omega^\circ$, by
\begin{equation}\label{eq:localizer-core-incidence-matrix}
 (B_\omega^\circ)_{s,(c,C)}
 :=
 \1_{\{C_\omega(cs,+)=C\}}
 -
 \1_{\{C_\omega(cs,-)=C\}},
\end{equation}
and set
$r_\omega^\circ:=\rank B_\omega^\circ$.

 {As in the proof of Lemma~\ref{lem:state-rank-dimension}, every
column of $B_\omega^\circ$ is nonzero. Since $A\ne\emptyset$,
\[
 r_\omega^\circ=0
 \quad\Longleftrightarrow\quad
 \mathscr C_\omega^\circ=\emptyset.
\]}
Lemma~\ref{lem:state-rank-dimension}, together with
\eqref{eq:localizer-admissible-pattern}, gives
\begin{equation}\label{eq:localizer-state-dimension}
 \sum_{C\in\mathscr C_\omega}d_\omega(C)
 \le
 2ar_\omega^\circ+2(J-1)b_F.
\end{equation}
All quantities in this inequality depend only on $\omega$, not on the
chosen pair producing it.

 {We now define a semialgebraic parameter space for the state values
associated with the fixed record $\omega$.} Let
\begin{equation}\label{eq:localizer-state-space}
 \mathcal S_\omega
 :=
 \left\{
  (z_C)_{C\in\mathscr C_\omega}
  \in
  \prod_{C\in\mathscr C_\omega}\Delta_{I_\omega(C)}:
  z_C\ne z_{C'}\ \text{whenever }C\ne C'
 \right\}.
\end{equation}
If $\mathscr C_\omega=\emptyset$, the product in
\eqref{eq:localizer-state-space} is understood to be a singleton.

 The
space $\mathcal S_\omega$ is semialgebraic. Indeed, every
$\Delta_{I_\omega(C)}$ is cut out by linear equalities and
inequalities, while $z_C\ne z_{C'}$ is equivalent to
$
 \sum_{i=1}^J(z_{C,i}-z_{C',i})^2>0.
$
Each coordinate $z_C$ belongs to $\Delta_{I_\omega(C)}$, whose
dimension is $d_\omega(C)$. Since
$\mathcal S_\omega\subseteq
\prod_{C\in\mathscr C_\omega}\Delta_{I_\omega(C)}$, it follows that (using \eqref{eq:localizer-state-dimension})
\begin{equation}\label{eq:localizer-state-space-dimension}
 \dim\mathcal S_\omega
 \le
 \sum_{C\in\mathscr C_\omega}d_\omega(C)\le 2ar_\omega^\circ+2(J-1)b_F.
\end{equation}

Define
\[
 J_\omega
 :=
 \begin{cases}
  \{0\},&\tau_\omega=\emptyset,\\
  (0,1],&\tau_\omega=(k,h)\text{ for some }(k,h),
 \end{cases}
\]
and set
\begin{equation}\label{eq:localizer-enlargement}
 Z_\omega:=\mathcal S_\omega\times J_\omega.
\end{equation}
A point of $Z_\omega$ will be denoted by
\[
 \zeta=(\mathbf z,\beta),
 \qquad
 \mathbf z=(z_C)_{C\in\mathscr C_\omega}.
\]
By 
\eqref{eq:localizer-state-space-dimension},
\begin{equation}\label{eq:localizer-pattern-dimension}
 \dim Z_\omega
 \le
 2ar_\omega^\circ+2(J-1)b_F+\dim J_\omega
 \le
 2ar_\omega^\circ+2(J-1)b_F+1.
\end{equation}

Define the affine map
\[
 \widehat f_*:\Delta^{J-1}\longrightarrow\mathbb R^m,
 \qquad
 \widehat f_*(z):=\sum_{i=1}^Jv_i z_i.
\]
Then $f_*=\widehat f_*\circ\boldsymbol{\psi}$.
For $\zeta=(\mathbf z,\beta)\in Z_\omega$, define the anchor matrix
$A_\omega(\beta)$, with rows indexed by $s\in F$ and columns indexed
by $(t,\nu)\in F\times\Lambda$, as follows. If
$\tau_\omega=\emptyset$, set
\begin{equation}\label{eq:localizer-abstract-anchor-no-hit}
 A_\omega(0)_{s,(t,\nu)}
 :=\1_{\{(t,\nu)=(s,\ell_\omega)\}}.
\end{equation}
If $\tau_\omega=(k_\omega,h_\omega)$, set
\begin{equation}\label{eq:localizer-abstract-anchor-hit}
 A_\omega(\beta)_{s,(t,\nu)}
 :=
 \1_{\{(t,\nu)=(s,\ell_\omega)\}}
 -
 \beta\,
 \1_{\{sh_\omega\in F\}}
 \1_{\{(t,\nu)=(sh_\omega,k_\omega)\}}.
\end{equation}
Define the Veronese matrix $V_\omega(\mathbf z)$, with rows indexed by
$s\in F$ and columns indexed by
$(c,\alpha)\in A\times\mathcal A_{J,d_V}$, by
\begin{equation}\label{eq:localizer-abstract-veronese-matrix}
 V_\omega(\mathbf z)_{s,(c,\alpha)}
 :=
 \sum_{C\in\mathscr C_\omega}
 \left(
  \1_{\{C_\omega(cs,+)=C\}}
  -
  \1_{\{C_\omega(cs,-)=C\}}
 \right)z_C^\alpha.
\end{equation}
Finally, define
\begin{equation}\label{eq:localizer-abstract-full-matrix}
 M_\omega(\zeta)
 :=
 \bigl[A_\omega(\beta)\mid V_\omega(\mathbf z)\bigr]
 \in\operatorname{Mat}_{|F|\times P}(\mathbb R),
\end{equation}
and, for $1\le j\le m$, define
$c_{\omega,j}:Z_\omega\to\mathbb R^F$ by
\begin{equation}\label{eq:localizer-abstract-right-hand-side}
 c_{\omega,j}(\zeta)_s
 :=
 \sum_{C\in\mathscr C_\omega}
 \left(
  \1_{\{C_\omega(s,-)=C\}}
  -
  \1_{\{C_\omega(s,+)=C\}}
 \right)
 \widehat f_*(z_C)_j.
\end{equation}
 {Classes occurring only at loop sites are absent from
$\mathscr C_\omega$ by definition. Moreover,} if
$C_\omega(u,+)=C_\omega(u,-)$, the two indicators in
\eqref{eq:localizer-abstract-veronese-matrix} and
\eqref{eq:localizer-abstract-right-hand-side} cancel.

The maps $M_\omega$ and $c_{\omega,j}$ are semialgebraic. Indeed,
\eqref{eq:localizer-abstract-anchor-no-hit} is constant,
\eqref{eq:localizer-abstract-anchor-hit} is affine in $\beta$,
\eqref{eq:localizer-abstract-veronese-matrix} is polynomial in the
coordinates of the $z_C$'s, and
\eqref{eq:localizer-abstract-right-hand-side} is affine in those
coordinates because $\widehat f_*$ is affine.

To relate these formulas to an actual pair, let
$(x,y)\in K_{\ell_\omega}\times(X\setminus O_{\ell_\omega})$ produce
$\omega$. For each $C\in\mathscr C_\omega$, let
$z_C^{x,y}$ be the common value of $\xi^{x,y}_{u,\sigma}$ for
$(u,\sigma)\in C$. These values are pairwise distinct and belong to
$\Delta_{I_\omega(C)}$. In the no-hit case   {denote} $\beta^{x,y}:=0$; in
a hit case let $\beta^{x,y}=\rho_{k_\omega}(v)$ be defined as in \eqref{eq:v}. Denote
\[
 \zeta^{x,y}
 :=
 \bigl((z_C^{x,y})_{C\in\mathscr C_\omega},\beta^{x,y}\bigr)
 \in Z_\omega.
\]
 {For $s\in F$, $c\in A$, and $\alpha\in\mathcal A_{J,d_V}$,
the defining sum for the Veronese block gives
\[
 V_\omega(\mathbf z^{x,y})_{s,(c,\alpha)}
 =\boldsymbol{\psi}(csx)^\alpha-\boldsymbol{\psi}(csy)^\alpha.
\]
Indeed, if the endpoint classes at $cs$ differ, the two indicators
select their respective values; if they agree, both sides vanish.
Likewise, the anchor formulas and the reversed order of the
indicators in the right-hand side give
\[
 A_\omega(\beta^{x,y})_{s,(t,\nu)}
 =b_{t,\nu}(sx)-b_{t,\nu}(sy),\qquad
 c_{\omega,j}(\zeta^{x,y})_s
 =f_*(sy)_j-f_*(sx)_j.
\]
Here $\mathbf z^{x,y}=(z_C^{x,y})_{C\in\mathscr C_\omega}$, and the
coordinate labeled $(c,\alpha)$ in $\Xi(sx)$ is
$\boldsymbol{\psi}(csx)^\alpha$. Hence the two blocks recover the coefficients of
the scalar collision equations above.}

By \eqref{eq:localizer-first-anchor},
\eqref{eq:localizer-anchor-row}, and the definitions of the endpoint
classes, the row of $M_\omega(\zeta^{x,y})$ indexed by $s$ is exactly
$\Xi(sx)-\Xi(sy)$,
while
$c_{\omega,j}(\zeta^{x,y})_s =f_*(sy)_j-f_*(sx)_j$.
Consequently, the collision equations
\begin{equation}\label{eq:localizer-collision}
 q_\Theta(sx)=q_\Theta(sy)
 \qquad(s\in F)
\end{equation}
are equivalent, in the $j$-th output coordinate, to
\begin{equation}\label{eq:localizer-linear-system}
 M_\omega(\zeta^{x,y})\theta^{(j)}
 =c_{\omega,j}(\zeta^{x,y}),
 \qquad 1\le j\le m.
\end{equation}

  {To apply Lemma~\ref{lem:semialgebraic-incidence-avoidance}, we
need the parameter-space dimension to be smaller than $m$ times
the rank of the collision matrix. We establish two lower bounds
for that rank.} From
\eqref{eq:localizer-abstract-anchor-no-hit}--
\eqref{eq:localizer-abstract-anchor-hit}, the anchor rows have exactly the
forms already analyzed in \eqref{eq:localizer-anchor-row}. If
$\tau_\omega=\emptyset$, they are the standard basis rows. If
$\tau_\omega=(k_\omega,h_\omega)$ and
$k_\omega\ne\ell_\omega$,   {retaining only the columns $(t,\ell_\omega)$, $t\in F$, gives
\[
 \bigl(A_\omega(\beta)_{s,(t,\ell_\omega)}\bigr)_{s,t\in F}
 =\bigl(\mathbf1_{\{s=t\}}\bigr)_{s,t\in F}=I_{|F|}.
\]
Indeed, none of the negative terms occurs in these columns. Thus
these rows are linearly independent.} If
$k_\omega=\ell_\omega$, then $h_\omega\ne e$, and
 {Lemma~\ref{lem:anchor-rank}, with $e_t$ identified with
$e_{(t,\ell_\omega)}$, gives rank at least $|F|/2$ for every}
$\beta\in(0,1]$. Hence, for every $\zeta\in Z_\omega$,
\begin{equation}\label{eq:localizer-full-anchor-rank}
 \rank A_\omega(\beta)\ge\frac{|F|}{2},
 \qquad
 \rank M_\omega(\zeta)\ge\frac{|F|}{2}.
\end{equation}

Since the rank of $M_\omega$ may vary, define its \textbf{rank strata}
by
\begin{equation}\label{eq:localizer-rank-stratum}
 Z_{\omega,r}:=\{\zeta\in Z_\omega:\rank M_\omega(\zeta)=r\},
 \qquad 0\le r\le |F|.
\end{equation}
The rank strata are semialgebraic by the minor conditions for rank, and form a
finite partition of $Z_\omega$ after empty pieces are omitted.
Since $\Omega$ is finite, there are at most $|\Omega|(|F|+1)$
nonempty strata in total. On each stratum we must prove
$\dim Z_{\omega,r}<mr$. The anchor estimate gives $r\ge |F|/2$;
we also need $r\ge r_\omega^\circ$ to control the term
$2ar_\omega^\circ$ in \eqref{eq:localizer-pattern-dimension}.
We obtain this second bound from the Veronese block.

Fix $\mathbf z\in\mathcal S_\omega$, and suppose that
$\lambda=(\lambda_s)_{s\in F}\in\mathbb R^F$  satisfies
$\lambda^{\mathsf T}V_\omega(\mathbf z)=0$. For every $c\in A$,
this gives the vector identity
\[
 \sum_{C\in\mathscr C_\omega}
 \left(
  \sum_{s\in F}\lambda_s
  \bigl(
   \1_{\{C_\omega(cs,+)=C\}}
   -
   \1_{\{C_\omega(cs,-)=C\}}
  \bigr)
 \right)
 \Phi_{d_V}(z_C)=0.
\]
There are at most $2|H|$ classes in $\mathscr C_\omega$, the points
$(z_C)_{C\in\mathscr C_\omega}$ are pairwise distinct by
\eqref{eq:localizer-state-space}, and $d_V\ge2|H|-1$. Therefore
Lemma~\ref{lem:finite-veronese} implies that the vectors
$\Phi_{d_V}(z_C)$, $C\in\mathscr C_\omega$, are linearly independent.
It follows  that
\[
 \sum_{s\in F}\lambda_s
 \bigl(
  \1_{\{C_\omega(cs,+)=C\}}
  -
  \1_{\{C_\omega(cs,-)=C\}}
 \bigr)
 =0
\]
for every $c\in A$ and every $C\in\mathscr C_\omega$. In particular,
using \eqref{eq:localizer-core-incidence-matrix},
$\lambda^{\mathsf T}B_\omega^\circ=0$.
Thus
\[
 \ker\bigl(V_\omega(\mathbf z)^{\mathsf T}\bigr)
 \subset
 \ker\bigl((B_\omega^\circ)^{\mathsf T}\bigr).
\]
For a $p\times q$ matrix $C$, the rank--nullity formula gives
$\rank C=p-\dim\ker C^{\mathsf T}$. Both matrices above have
$|F|$ rows, so
\begin{equation}\label{eq:localizer-veronese-rank}
\begin{aligned}
 \rank V_\omega(\mathbf z)
 &{}=|F|-\dim\ker V_\omega(\mathbf z)^{\mathsf T}\\
 &{}\ge |F|-\dim\ker(B_\omega^\circ)^{\mathsf T}\\
 &{}=\rank B_\omega^\circ=r_\omega^\circ.
\end{aligned}
\end{equation}
Since $M_\omega(\zeta)$ contains both the anchor and Veronese column
blocks, if
$
 r:=\rank M_\omega(\zeta),
$
then \eqref{eq:localizer-full-anchor-rank} and
\eqref{eq:localizer-veronese-rank} give
\begin{equation}\label{eq:localizer-two-ranks}
 r\ge\max\left\{r_\omega^\circ,\frac{|F|}{2}\right\}.
\end{equation}

 {Fix a nonempty rank stratum $Z_{\omega,r}$. By
\eqref{eq:localizer-two-ranks}, $r\ge r_\omega^\circ$ and
$r\ge |F|/2$. The first inequality controls the main term in
\eqref{eq:localizer-pattern-dimension}; the second and
\eqref{eq:localizer-boundary-budget} give
\[
 2ar_\omega^\circ\le2ar,
 \qquad
 2(J-1)b_F+1<\frac{m-2a}{2}|F|\le(m-2a)r.
\]
Here $a>0$ follows from \eqref{eq:localizer-A-order}, since the
face dimensions are nonnegative, and $m-2a>0$. Consequently,
\begin{equation}\label{eq:localizer-dimension-rank-gap}
\begin{aligned}
 \dim Z_{\omega,r}
 &{}\le 2ar_\omega^\circ+2(J-1)b_F+1\\
 &{}<2ar+(m-2a)r=mr.
\end{aligned}
\end{equation}}

For $\omega\in\Omega$ and every nonempty rank stratum
$Z_{\omega,r}$, define
\begin{equation}\label{eq:localizer-exceptional-set}
 \begin{aligned}
 \mathcal B_{\omega,r}
 :=\bigl\{\Theta\in(\mathbb R^P)^m: &\ \exists
 \zeta\in Z_{\omega,r}\text{ such that }\\
 &M_\omega(\zeta)\theta^{(j)}=c_{\omega,j}(\zeta)~
 \forall 1\le j\le m\bigr\}.
 \end{aligned}
\end{equation}
By Lemma~\ref{lem:semialgebraic-incidence-avoidance} and
\eqref{eq:localizer-dimension-rank-gap}, each
$\mathcal B_{\omega,r}$ is semialgebraic and has nowhere dense closure
in $(\mathbb R^P)^m$. Since $\Omega$ is finite and each
$M_\omega$ has only finitely many possible ranks, the set
\begin{equation}\label{eq:localizer-total-exceptional-set}
 \mathcal B
 :=
 \bigcup_{\omega\in\Omega}
 \bigcup_{\substack{0\le r\le|F|\\Z_{\omega,r}\ne\emptyset}}
 \mathcal B_{\omega,r}\text{ has nowhere dense closure.}
\end{equation}

Suppose that  there
are $\ell\in\Lambda$, $x\in K_\ell$, and
$y\in X\setminus O_\ell$ such that
$q_\Theta^F(x)=q_\Theta^F(y)$. The pair determines the record
$\omega=\omega(\ell,x,y)\in\Omega$ and the point
$\zeta^{x,y}\in Z_\omega$ constructed above. Let
$r=\rank M_\omega(\zeta^{x,y})$. By
\eqref{eq:localizer-linear-system},
$\Theta\in\mathcal B_{\omega,r}\subset\mathcal B$.
Consequently, if $\Theta\notin\mathcal B$, then no such actual
collision occurs, and hence
\[
 q_\Theta^F(K_\ell)\cap q_\Theta^F(X\setminus O_\ell)=\emptyset
 \qquad(\ell\in\Lambda).
\]
Since $\overline{\mathcal B}$ has empty interior, every neighbourhood of
the origin in $(\mathbb R^P)^m$ contains a parameter outside
$\mathcal B$.

Finally, $f_*(X)$ is a compact subset of $(0,1)^m$, and $\Xi(X)$
is bounded. Hence $q_\Theta\to f_*$ uniformly as $\Theta\to0$. We
may therefore choose $\Theta\notin\mathcal B$ sufficiently close to the
origin that
\begin{equation}\label{eq:localizer-final-choice}
 \|q_\Theta-f_*\|_\infty<\varepsilon
 \qquad\text{and}\qquad
 q_\Theta(X)\subset[0,1]^m.    
\end{equation}
Set $q:=q_\Theta$. Since $e\in A$, the affine term $f_*(z)$
is determined by $\boldsymbol{\psi}^{A}(z)$. The formula
\eqref{eq:localizer-parametric-map} therefore defines a continuous map
\[
 Q:(\Delta^{J-1})^A\times[0,1]^{F^{-1}\times\Lambda}
 \longrightarrow\mathbb R^m
\]
such that
\[
 q(z)=Q\bigl(
   \boldsymbol{\psi}^{A}(z),
   \boldsymbol{\rho}^{F^{-1}}(z)
 \bigr).
\]
Together with \eqref{eq:localizer-bump-functions}, this also proves
the assertions of Remark~\ref{rem:prepared-localizer}.
\end{proof}

\bibliographystyle{alpha}
\bibliography{ref}

\end{document}